\documentclass{amsart}
\usepackage{latexsym,amsfonts,amsmath,amssymb}
\usepackage{hyperref}

\newtheorem{theorem}{Theorem}[section]
\newtheorem{proposition}[theorem]{Proposition}
\newtheorem{lemma}[theorem]{Lemma}
\newtheorem{corollary}[theorem]{Corollary}

\theoremstyle{definition}
\newtheorem{definition}[theorem]{Definition}

\theoremstyle{remark}
\newtheorem*{remark}{Remark}
\newtheorem*{remarks}{Remarks}

\newcommand{\Z}{{\mathbb{Z}}}
\newcommand{\C}{{\mathbb{C}}}
\newcommand{\R}{{\mathbb{R}}}

\newcommand{\E}{{\mathbb{E}}}
\newcommand{\HH}{{\mathbb{H}}}
\newcommand{\OO}{{\mathcal O}}
\newcommand{\ld}{\Delta_{\Omega}}
\newcommand{\tpo}{\tilde{P}^{\Omega}}
\newcommand{\po}{P^{\Omega}}
\newcommand{\pon}{\tilde{P}_N^{\Omega}}
\newcommand{\lpr}{L^p(\R^d)}
\newcommand{\lpo}{L^p(\Omega)}
\newcommand{\tz}{2^{\Z}}
\newcommand{\eps}{\varepsilon}

\DeclareMathOperator*{\diam}{diam}
\DeclareMathOperator*{\dist}{dist}
\DeclareMathOperator*{\supp}{supp}

\numberwithin{equation}{section}

\newcommand{\qtq}[1]{\quad\text{#1}\quad}

\title{Riesz transforms outside a convex obstacle}

\author[R. Killip]{Rowan Killip}
\address{Department of Mathematics, UCLA}%
\email{killip@math.ucla.edu}

\author[M. Visan]{Monica Visan}
\address{Department of Mathematics, UCLA}
\email{visan@math.ucla.edu}

\author[X. Zhang]{Xiaoyi Zhang}
\address{Department of Mathematics, University of Iowa, and Chinese Academy of Science, Beijing}%
\email{zh.xiaoyi@gmail.com}

\begin{document}

\begin{abstract}
The goal of this paper is to develop some basic harmonic analysis tools for the Dirichlet Laplacian in the exterior domain associated to a smooth convex obstacle in dimensions $d\geq 3$.  Specifically, we will discuss analogues   of the Mikhlin Multiplier Theorem, Littlewood--Paley Theory, and Hardy inequalities, culminating in a proof that homogeneous Sobolev norms defined with respect to the Dirichlet and whole-space Laplacians are equivalent for the sharp ranges of integrability exponent $p$ and regularity $s$.   Counterexamples are included to show that these results are indeed sharp.  In particular, we precisely settle the question of boundedness of Riesz transforms on $L^p$, including the endpoint.

The utility of such results in the study of nonlinear PDE is that they allow us to deduce important results, such as the fractional product and chain rules for the Dirichlet Laplacian, directly from the classical Euclidean setting.  As an application, we discuss the local well-posedness and stability problems for energy-critical NLS.  All the results of this paper play an essential role in the authors' proof of large-data global well-posedness and scattering for the energy-critical NLS in three dimensional exterior domains; see \cite{KVZ:convex}.
\end{abstract}

\maketitle

\section{Introduction}

Throughout this paper, $\Omega$ will denote the complement of a compact convex body $\Omega^c\subset\R^d$ with smooth boundary and $d\geq 3$.  We consider the Dirichlet Laplacian on $\Omega$, which we denote by $-\Delta_\Omega$.  For functions $f\in C^\infty_c(\Omega)$ the action of the Laplacian is unambiguous: $\Delta_\Omega f (x)  = \sum \partial_j^2 f(x)$.  For more general functions, the boundary condition plays a role.

Let us define $H_D^1(\Omega)$, where $D$ stands for Dirichlet, as the completion of $C^\infty_c(\Omega)$ with respect to the inner product
$$
\langle f,\;g\rangle_{H_D^1(\Omega)}  = \int_\Omega \overline{\nabla f(x)} \cdot \nabla g(x)  + \overline{f(x)} g(x)\,dx.
$$
Note that this is exactly how the classical Sobolev space $H_0^1(\Omega)$ is defined; the spaces are identical.  However, below we will see how different
natural generalizations of these spaces lead to norms that are not always equivalent; indeed, this is a central theme of this paper.

As the $H^1_D(\Omega)$ norm controls the $L^2(\Omega)$ norm (and $C^\infty_c(\Omega)\subseteq L^2(\Omega)$), we can naturally view $H^1_D(\Omega)$ as a subset of $L^2(\Omega)$.  By construction,
$$
Q: f \mapsto \int_\Omega \overline{\nabla f(x)} \cdot \nabla f(x) \,dx
$$
defines a closed quadratic form on $L^2(\Omega)$ with domain $H_D^1(\Omega)$.  Thus, there is a unique (unbounded) non-negative self-adjoint operator on $L^2(\Omega)$ for which $Q$ is the associated quadratic form.  This operator is the Dirichlet Laplacian  $-\Delta_\Omega$.   The space $H_D^1(\Omega)$ is the domain of the square root of this operator. (For these deductions and much more, see \cite[Ch.~VIII]{RS1} or \cite[Ch.~6]{Kato:pert}.)

From the spectral theorem, one can then define general functions of the operator $-\Delta_\Omega$, at least as (possibly unbounded) operators on $L^2(\Omega)$.  In this way one obtains solutions to heat, wave, Schr\"odinger, Poisson, and other equations, at least in some abstract Hilbert-space sense.  We discuss briefly how these solutions can be interpreted more concretely, taking the heat equation as our basic model; we will at least see the sense in which $-\Delta_\Omega$ is the Dirichlet Laplacian.

Given $f\in C^\infty_c(\Omega)$, the spectral theorem guarantees that $u(t):=e^{t\Delta_\Omega} f$ is a $C^\infty$ function of time $t\in[0,\infty)$ with values in $H^1_D(\Omega)$.  By testing against general functions $G\in C^\infty_c([0,\infty)\times\Omega)$ and $g\in C^\infty_c(\Omega)$, one quickly sees that $u$ is distributional solution to the heat equation and that $u(t)$ and $\partial_t u(t)$ belong to the classical Sobolev spaces $H^{k}(\Omega)$ for all integers $k\geq 0$.  (Recall that $H^{k}(\Omega)=W^{k,2}(\Omega)$ is defined as the set of functions $f\in L^2(\Omega)$ whose distributional derivatives of order $\leq k$ belong to $L^2(\Omega)$.)  Thus $u$ is a (smooth) classical solution to the heat equation on $[0,\infty)\times\Omega$.

To see that $u$ vanishes on $\partial\Omega$, we may either invoke the classical trace theorem for $H_0^1(\Omega)=H^1_D(\Omega)$ or combine the regularity described above with the Hardy inequality (cf. Lemmas~\ref{L:Hardy Laplacian} and~\ref{L:Hardy Dirichlet}).  Either way, the maximum principle guarantees that there is a unique bounded solution to the heat equation vanishing on $\partial\Omega$ with initial data $f\in C^\infty_c(\Omega)$ and so we may refer to $e^{t\Delta_\Omega}$ as \emph{the} fundamental solution to the heat equation with Dirichlet boundary conditions.  The corresponding heat kernel $e^{t\Delta_\Omega}(x,y)$ will play a key role in the discussion below.

There is a natural family of Sobolev spaces associated to powers of the Dirichlet Laplacian.  Our notation for these is as follows:

\begin{definition}
For $s\geq 0$ and $1<p<\infty$, let $\dot H_D^{s,p}(\Omega)$ and $H_D^{s,p}(\Omega)$ denote the completions of $C^\infty_c(\Omega)$ under the norms
$$
\| f \|_{\dot H_D^{s,p}(\Omega)} := \| (-\Delta_\Omega)^{s/2} f \|_{L^p} \qtq{and} \| f \|_{H_D^{s,p}(\Omega)} := \| (1-\Delta_\Omega)^{s/2} f \|_{L^p}.
$$
For a proof that $(-\Delta_\Omega)^{s/2} f \in L^p(\Omega)$ for $f\in C^\infty_c(\Omega)$, see the proof of Theorem~\ref{T:sq}.  When $p=2$ we write $\dot H^s_D(\Omega)$ and $H^s_D(\Omega)$ for $\dot H^{s,2}_D(\Omega)$ and $H^{s,2}_D(\Omega)$, respectively.
\end{definition}

When $\Omega$ is replaced by $\R^d$, these definitions lead to the classical $\dot H^{s,p}(\R^d)$ and $H^{s,p}(\R^d)$ families of spaces.  For non-integer $s$, these are distinct from the $\dot W^{s,p}(\R^d)$ and $W^{s,p}(\R^d)$ spaces.

We warn the reader that the paper \cite{SmithSogge} uses the very similar notation $\dot H_D^{s}(\Omega)$ for another notion of Dirichlet Sobolev space (in the $p=2$ setting).  The authors do not discuss the connection to powers of the Dirichlet Laplacian; however the results of this paper show that the two definitions (here and in \cite{SmithSogge}) coincide for $s<3/2$.  Simpler arguments to this effect are available in the regime $0<s\leq 1$, exploiting the equivalence at $s=1$, which is self-evident.

Sobolev spaces on domains are usually introduced as subspaces/quotients of the corresponding spaces on $\R^d$, rather than via the Dirichlet (or Neumann) Laplacian:

\begin{definition}
Fix $s\geq 0$ and $1<p<\infty$.  The space $\dot H^{s,p}(\Omega)$ consists of the restrictions of functions in $\dot H^{s,p}(\R^d)$ to $\Omega$, together with the norm
$$
\|f\|_{\dot H^{s,p}(\Omega)}=\inf \bigl\{ \|g\|_{\dot H^{s,p}(\R^d)}:\, g|_{\Omega}=f \bigr\}.
$$
The space $\dot H^{s,p}_0(\Omega)$ is defined as the completion of $C^\infty_c(\Omega)$ in $\dot H^{s,p}(\Omega)$, while the space $\dot H^{s,p}_{00}(\Omega)$ is defined as the completion of $C^\infty_c(\Omega)$ in $\dot H^{s,p}(\R^d)$.   The inhomogeneous spaces $H^{s,p}(\Omega)$, $ H^{s,p}_0(\Omega)$, and $H^{s,p}_{00}(\Omega)$ are defined analogously.
\end{definition}

The notation $H^{s,p}_{00}(\Omega)$ is taken from the works of Lions--Magenes (cf. \cite{LionsMagenes:vol1}), where it was introduced in a different manner (and only when $p=2$), but later shown to coincide with this definition.
There is a subtle difference between the spaces $H^{s,p}_{00}$ and $H^{s,p}_{0}$ even in the case of a half-space.   Both are based on extensions to all of $\R^d$ and the $H^{s,p}(\R^d)$ norm; however, in the former case one must extend the function as zero, while in the latter one may extend the function more generally (e.g. via odd reflection).  For a smooth bounded domain $\OO$, it is well known that $H^{s,p}_{00}(\OO) = H^{s,p}_{0}(\OO)$ except when $s-\frac1p$ is an integer.  When $s-\frac1p$ is an integer, there is strict containment.

One of the key questions we address is whether
\begin{equation}\label{E:our question}
\dot H^{s,p}_{00}(\Omega) = \dot H^{s,p}_{D}(\Omega), \quad\text{in the sense of equivalent norms?}
\end{equation}
That we are considering \emph{homogeneous} spaces lies at the heart of this question.  These are the natural spaces when studying scale-invariant PDE such as the energy-critical NLS discussed in Section~\ref{SS:NLS}. The treatment of inhomogeneous spaces is rather different; the additional decay guaranteed by the $L^p$ bound allows one to localize, essentially reducing matters to the case of bounded domains.  The subtleties of question \eqref{E:our question} lie in the behaviour at large length scales.

The analogue of question \eqref{E:our question} for smooth bounded domains $\OO$ has been completely settled.  Note that in that case, $\dot H^{s,p}_{00}(\OO)= H^{s,p}_{00}(\OO)$; we will use the latter notation when reviewing the known results in this direction.

For smooth bounded domains, much effort has been expended on the explicit characterization of the spaces $H_D^{s,p}(\OO)$; see in particular the papers
\cite{Fujiwara,Grisvard,Seeley}, which also cover higher-order operators and those with non-constant coefficients, as well as \cite{Seeley:ICM}, for an introduction to this work. Note that the spaces $H_D^{s,p}(\OO)$ can also be identified as the domains of powers of $-\Delta_\OO$ acting on $L^p(\OO)$.  Moreover, as imaginary powers $(-\Delta_\OO)^{it}$ are $L^p$-bounded (by a polynomial in $t$), one may also interpret $H_D^{s,p}(\OO)$ as complex interpolation spaces.  From the papers just referenced we find the following:
$$
H_D^{s,p}(\OO) = \begin{cases} H_{00}^{s,p}(\OO) &: 0\leq s < 1 + \tfrac1p \\[0.5ex]   H_{00}^{1,p}(\OO)\cap H^{s,p}(\OO) &: 1 + \tfrac1p \leq s \leq 2 \end{cases} 	\quad\text{on smooth bounded domains}.
$$
More accurately, to phrase the results in this manner one needs to combine those papers with well-known  inter-relations between the classical Sobolev spaces, such as $H_{00}^{s,p}(\OO)=H_{0}^{s,p}(\OO)=H^{s,p}(\OO)$ when $0\leq s < \frac1p$.  We should also note that it follows that $H_D^{s,p}(\OO) \neq H_{00}^{s,p}(\OO)$ when $s\geq 1 + \tfrac1p$.  (For a simple direct argument to this effect, see Proposition~\ref{P:1+1/p}.)

When $s=1$, \eqref{E:our question} already contains the question of $L^p$-boundedness of Reisz transforms, that is, whether $\nabla (-\Delta_\Omega)^{-1/2}:L^p\to L^p$.  This is a topic of intensive investigation; see, for example, \cite{ACDH,Bakry,BorchMiya,CCH,CoulhonDuong,GuiHassell,HassellSikora,LSZ,Strichartz:Laplacian}, as well as the many references therein.  Much of this work focuses on the case of smooth boundary-less non-compact manifolds.  While such general results do strongly suggest what may hold in the case of exterior domains, it seems that only the papers \cite{BorchMiya,HassellSikora,LSZ} single this case out for particular attention.  One achievement of these earlier works has been to completely resolve the question of boundedness of Riesz transforms in the case of spherically symmetric functions in the exterior of the unit ball.  Specifically, in \cite{HassellSikora} it is shown that the Riesz transforms are bounded for $d=2$ and $p\leq 2$ and also for $d\geq 3$ and $p<d$.  When $d=3$ this result was also proved independently in \cite{LSZ}.

For the problem investigated in \cite{KVZ:convex}, one needs the generalization to some $s> 1$ and some $s<1$ in order to properly describe the behaviour of high- and low-frequency portions of the solution.  The Littlewood--Paley theory described in Section~\ref{SS:LP} is used extensively in \cite{KVZ:convex} as a means to divide up the solution into such high- and low-frequency portions.

As a culmination of the tools developed in this paper, we will prove the following:

\begin{theorem}\label{T:main}
Let $d\geq 3$ and $\Omega$ the complement of a compact convex body $\Omega^c\subset\R^d$ with smooth boundary.  Suppose $1<p<\infty$ and $0\leq s<\min\{1+\frac1p,\frac dp\};$ then
\begin{equation}\label{E:equiv norms}
\bigl\| (-\Delta_{\R^d})^{s/2} f \bigl\|_{L^p}  \sim_{d,p,s} \bigl\| (-\Delta_\Omega)^{s/2} f \bigr\|_{L^p}  \qtq{for all} f\in C^\infty_c(\Omega).
\end{equation}
Thus $\dot H_D^{s,p}(\Omega)=\dot H^{s,p}_{00}(\Omega)$ for these values of the parameters.
\end{theorem}

In particular, by combining this with a density argument we will deduce the following generalization of the boundedness of Riesz transforms:

\begin{corollary}\label{C:Riesz}
Let $d\geq 3$ and $\Omega$ the complement of a compact convex body $\Omega^c\subset\R^d$ with smooth boundary.  Suppose $1<p<\infty$ and $0\leq s<\min\{1+\frac1p,\frac dp\};$ then
\begin{equation}\label{E:Riesz}
\bigl\| (-\Delta_{\R^d})^{s/2} (-\Delta_\Omega)^{-s/2} f \bigl\|_{L^p(\R^d)}  \lesssim \bigl\| f \bigr\|_{L^p(\Omega)}.
\end{equation}
\end{corollary}

In Section~\ref{SS:counter}, we will prove that the restriction on $s$ is sharp (including the question of endpoints).  When $s=1$ this sharpness was noted in earlier work on spherically symmetric functions in the exterior of the unit ball.  Specifically, \cite{HassellSikora} shows that Riesz transforms are unbounded when $p\geq d \geq 3$ and when $p>d=2$; see also the independent paper \cite{LSZ}, which treated $p>d=3$.  With regard to such counterexamples, we also draw attention to an earlier example \cite[\S 5]{CoulhonDuong} showing unboundedness of Riesz transforms for $p>d$ on the connected sum of two copies of $\R^d$.

As examples of the utility of Theorem~\ref{T:main}, we note that it gives the following corollaries of the existing product and chain rules from the Euclidean setting.  For the Euclidean results see \cite{ChW:fractional chain rule}, as well as \cite{Taylor:book} for a textbook treatment.

\begin{corollary}[Fractional product rule]\label{C:product rule}
For all $f, g\in C_c^{\infty}(\Omega)$, we have
\begin{align*}
\| (-\Delta_\Omega)^{\frac s2}(fg)\|_{L^p(\Omega)} \lesssim \| (-\Delta_\Omega)^{\frac s2} f\|_{L^{p_1}(\Omega)}\|g\|_{L^{p_2}(\Omega)}+
\|f\|_{L^{q_1}(\Omega)}\| (-\Delta_\Omega)^{\frac s2} g\|_{L^{q_2}(\Omega)}
\end{align*}
with the exponents satisfying $1<p, p_1, q_2<\infty$, $1<p_2,q_1\le \infty$,
\begin{align*}
\tfrac1p=\tfrac1{p_1}+\tfrac1{p_2}=\tfrac1{q_1}+\tfrac1{q_2}, \qtq{and} 0<s<\min\bigl\{ 1+\tfrac1{p_1}, 1+\tfrac1{q_2},\tfrac d{p_1},\tfrac d{q_2} \bigr\}.
\end{align*}
\end{corollary}

\begin{corollary}[Fractional chain rule]\label{C:chain rule}
Suppose $G\in C^1(\mathbb C)$, $s \in (0,1]$, and $1<p,p_1,p_2<\infty$ are such that $\frac 1p=\frac 1{p_1}+\frac 1{p_2}$ and
$0<s<\min\{1+\tfrac1{p_2},\tfrac d{p_2}\}$.   Then
$$
\|(-\Delta_\Omega)^{\frac s2}G(f)\|_{L^p(\Omega)}\lesssim \|G'(f)\|_{L^{p_1}(\Omega)}\| (-\Delta_\Omega)^{\frac s2} f\|_{L^{p_2}(\Omega)},
$$
uniformly for $f\in C^\infty_c(\Omega)$.
\end{corollary}

We prove Theorem~\ref{T:main} by estimating the difference between the Littlewood--Paley square functions associated to the two operators.
More precisely, we use Littlewood--Paley projections built from the heat semigroup rather than from a partition of unity adapted to
dyadic shells.  This leads us to estimate the difference between the two heat kernels, for which we employ the maximum principle.
As our obstacle is convex we are able to use the halfspace as a comparison domain, which simplifies this step.  (See Lemma~\ref{L:kernel}
for further details.)

The difference between the heat kernels $e^{t\Delta_{\R^d}}$ and $e^{t\Delta_\Omega}$ is concentrated near the obstacle.  Thus by
following the steps outlined above, we find that the difference between the two sides of \eqref{E:equiv norms} can be controlled by
an appropriately weighted $L^p$-norm of $f$.  More precisely, Proposition~\ref{P:diff} shows that
$$
\bigl\| (-\Delta_{\R^d})^{s/2} f \bigr\|_{L^p} + \biggl\|\frac{f(x)}{\dist(x,\Omega^c)^s}\biggr\|_{L^p}
    \sim \bigl\| (-\Delta_\Omega)^{s/2} f \bigr\|_{L^p} + \biggl\|\frac{f(x)}{\dist(x,\Omega^c)^s}\biggr\|_{L^p}
$$
for all $f\in C^\infty_c(\Omega)$. This relation holds for any $s> 0$ and any $1<p<\infty$.

From the preceding discussion, the proof of \eqref{E:equiv norms} reduces to proving Hardy-type inequalities.  This is done
in Section~\ref{SS:Hardy} (see Lemmas~\ref{L:Hardy Laplacian} and \ref{L:Hardy Dirichlet}) and is the source of the restriction $s<\min\{1+1/p, d/p\}$.

In the next section, we will discuss some consequences of Theorem~\ref{T:main} in the study of nonlinear Schr\"odinger equations.  We
devote the remainder of this section to describing some basic tools that enter into the proof of Theorem~\ref{T:main}.

From the maximum principle (or the Brownian motion interpretation) it is easy to see that the heat kernel associated to the Dirichlet Laplacian on $\Omega$
is both positive and majorized by the heat kernel for Euclidean space:
\begin{equation}\label{E:crude heat bound}
 0 \leq e^{t\Delta_\Omega}(x,y) \leq e^{t\Delta_{\R^d}}(x,y) \lesssim t^{-\frac d2} \exp\bigl( - |x-y|^2 /4t \bigr)
\end{equation}
for all $t>0$.  In Section~\ref{SS:multiplier} we will prove that when coupled with finite speed of propagation for the wave equation, this crude bound suffices to prove the analogue of the Mikhlin multiplier theorem for functions of $-\Delta_\Omega$; see Theorem~\ref{T:Mikhlin}.  The needed Littlewood--Paley square function estimates then follow directly from the usual argument; see Theorem~\ref{T:sq}.

The paper \cite{DuongOuhabazSikora} gives a very general multiplier theorem for positive self-adjoint operators whose semigroups obey Gaussian bounds.
This subsumes Theorem~\ref{T:Mikhlin}.  For completeness, we have chosen to include a proof of Theorem~\ref{T:Mikhlin}; it also gives us the opportunity to exhibit in the present setting some beautiful arguments developed in the context of multiplier theorems on Lie groups.

Littlewood--Paley square function estimates for exterior domains were also developed in \cite{IvanPlanch:square} and then used in \cite{IvanPlanch:IHP} to prove multilinear estimates in Besov spaces.  A second ingredient in the proof of their multilinear estimates is a weaker variant of boundedness of Riesz transforms, namely,
\begin{equation}\label{E:HKR'}
\bigl\| \sqrt{t} \, \nabla e^{t\Delta_\Omega} f \bigr\|_{L^p(\Omega)} \lesssim \|  f \|_{L^p(\Omega)} \quad\text{uniformly for $t>0$.}
\end{equation}
In \cite[Proposition~1.7]{IvanPlanch:square} it is claimed that this holds for all $1<p<\infty$.  We dispute this.  In Section~\ref{S:HKR}, we prove that \eqref{E:HKR'}
holds for $p<d$, but fails for $p>d$. 

The main result of \cite{IvanPlanch:IHP} is global well-posedness of the energy-critical NLS in non-trapping domains for small data.  A second proof of this result was given in \cite{BSS}.  This second paper works only with (scaling) inhomogeneous spaces and avoids subtle multilinear estimates by proving a $L^4_tL^\infty_x$ Strichartz estimate, which allows those authors to estimate the nonlinearity in $L^1_t H^1_x$.

The proof of the Hardy inequalities for the Dirichlet Laplacian (cf. Lemma~\ref{L:Hardy Dirichlet}) requires finer information about the heat kernel than is given in \eqref{E:crude heat bound}, which is oblivious to the shape of $\Omega$.  Much effort has been devoted to the precise estimation of heat kernels in various Riemannian manifolds. For the particular case of the exterior of a convex obstacle, the exact order of magnitude is a consequence of a more general result of Q.~S.~Zhang \cite{qizhang}; see also \cite{GrigoSaloff}.

\begin{theorem}[Heat kernel estimate, \cite{qizhang}]\label{T:heat}
There exists $c>0$ such that
\begin{align*}
|e^{t\Delta_\Omega}(x,y)| \lesssim \Bigl[\frac{\dist(x,\partial\Omega)}{\sqrt t \wedge \diam(\Omega^c)} \wedge 1\Bigr]
    \Bigl[\frac{\dist(y,\partial\Omega)}{\sqrt t\wedge \diam(\Omega^c)}\wedge 1\Bigr]
    t^{-\frac d 2} \exp\bigl( -c |x-y|^2 /t \bigr)
\end{align*}
uniformly for $x, y\in \Omega$ and $t>0$; recall that $A\wedge B = \min\{A,B\}$.  Moreover, the reverse inequality holds after suitable modification
of $c$ and the implicit constant.
\end{theorem}

Integrating these bounds in time with power-like weights, we obtain upper bounds on the kernels of $(-\Delta_\Omega)^{-s/2}$; see
Lemma~\ref{L:Riesz}.  These bounds are then used to prove the Hardy inequality for $-\Delta_\Omega$ in Lemma~\ref{L:Hardy Dirichlet}.
All the ingredients are finally brought together in Section~\ref{SS:equiv} to complete the proof of Theorem~\ref{T:main}.

As $\R^2$ is parabolic (the Green function is not sign definite), the correct order of magnitude for the heat kernel in planar exterior domains has a different structure to that given in Theorem~\ref{T:heat}.  Although we know of no reference that gives the full analogue of Theorem~\ref{T:heat} for $d=2$, the paper \cite{GrigoSaloff} does give the answer in the case of the exterior of the unit disk in $\R^2$; see Example~1.3 of that paper.  We do not know of any obstruction to adapting the methodology we use in this paper to the planar case.  Nevertheless, equivalence of Sobolev spaces in two dimensional exterior domains remains an interesting problem.

\subsection*{Acknowledgements}
We would like to thank Andrew Hassell for bringing several references to our attention.  R. K. was supported by NSF grant DMS-1001531.  M. V. was supported by the Sloan Foundation and NSF grants DMS-0901166 and DMS-1161396.  X. Z. was supported by the Sloan Foundation.

\section{Applications to NLS}\label{SS:NLS}

We were led to the subject of this paper by our investigations of the nonlinear Schr\"odinger equation in the domain $\Omega$ with Dirichlet boundary conditions, which one may write naively as follows:
\begin{equation}\label{E:'D'nls}
i \partial_t u = -\Delta u \pm |u|^p u \qtq{with} u(t=0,x)=u_0(x) \qtq{and} u(t)\bigr|_{\partial\Omega} \equiv 0.
\end{equation}
A more precise formulation (cf. Duhamel's principle) is to write the corresponding integral equation:
\begin{equation}\label{E:duhamel}
u(t) = e^{it\Delta_\Omega} u_0  \mp i \int_0^t e^{i(t-s)\Delta_\Omega} |u(s)|^p u(s)\, ds,
\end{equation}
where $\Delta_\Omega$ denotes the Dirichlet Laplacian.  (This gives the correct interpretation of the boundary condition, even for
solutions without meaningful restrictions to $\partial\Omega$.)  Together with reasonable assumptions on $u$, this leads directly to the notion of strong solutions; see Definition~\ref{D:solution} for a concrete example.

Local existence and uniqueness of strong solutions to nonlinear Schr\"odinger equations are usually proved by applying the contraction mapping method to this integral
equation.  A key ingredient in doing this is the family of estimates for the propagator $e^{it\Delta_\Omega}$ known as Strichartz estimates.  These estimates have a long
history and have been the subject of intensive investigation in Euclidean space, on manifolds, and in domains.  For the case of exterior domains discussed in this paper,
they were proved by Ivanovici \cite{Ivanovici:Strichartz}; see also \cite{BSS,IvanPlanch:IHP}.

\begin{theorem}[Strichartz estimates, \cite{Ivanovici:Strichartz}]\label{T:Strichartz}
Let $d\geq 2$, let $\Omega\subset\R^d$ be the exterior of a smooth compact strictly convex obstacle,
and let $q,\tilde q > 2$ and $2\leq r,\tilde r \leq \infty$ satisfy the scaling conditions
$$
\tfrac2{q} + \tfrac{d}{r} = \tfrac d2 = \tfrac2{\tilde q} + \tfrac{d}{\tilde r}.
$$
Then
$$
\biggl\| e^{it\Delta_\Omega} u_0 \mp i \int_0^t e^{i(t-s)\Delta_\Omega} F(s)\, ds\biggr\|_{L_t^qL_x^r(I\times\Omega)}
\lesssim \|u_0\|_{L^2(\Omega)} + \|F\|_{L_t^{\tilde q'}L_x^{\tilde r'}(I\times\Omega)},
$$
with the implicit constant independent of the time interval $I\ni0$.
\end{theorem}

These inequalities also hold for the Schr\"odinger equation in $\R^d$.  Indeed, in that setting, there is a very simple proof based on the explicit formula for the propagator and the Hardy--Littlewood--Sobolev inequality (in the time variable).  More precisely, the proof uses the dispersive estimate
$$
\bigl\| e^{it\Delta_{\R^d}} f \bigr\|_{L^\infty(\R^d)} \lesssim |t|^{-\frac{d}2} \bigl\| f \bigr\|_{L^1(\R^d)} \quad\text{for all $t\neq 0$}
$$
and the conservation of $L^2(\R^d)$, but no other information about the propagator.  In \cite{KeelTao}, it has been shown that one may
allow $q=2$ and/or $q'=2$ in the Euclidean setting when $d\geq 3$.  The argument is  significantly more complicated, but only uses the same information about the propagator: the dispersive estimate and conservation of $L^2$.

It is currently unknown whether the dispersive estimate holds outside a smooth convex obstacle.  In fact, it is not known to hold even in the exterior of a sphere, except in the class of radial functions, in which case it is a theorem of Li, Smith, and Zhang, \cite{LSZ}.  Geometric optics suggests very strongly that the dispersive estimate should be true.  At the same time, one may use parametrix methods inspired by geometric optics to see that the dispersive estimate must fail outside typical non-convex obstacles --- imagine light refocusing off a concave mirror.

The dispersive estimate plays a key role in the existing analyses of the long-time behaviour of solutions to the nonlinear Schr\"odinger equation in the mass-critical, energy-critical, and energy-supercritical regimes in the Euclidean setting.  Consider for example the global well-posedness problem for the energy-critical equation in three dimensions, for which the nonlinearity is quintic (i.e., $p=4$).  For this equation, large data global well-posedness has been shown to hold in the defocusing case; in the focusing case, global well-posedness is known for spherically symmetric solutions up to the soliton threshold.  In fact, not only are the solutions global in time, but they have finite $L^{10}_{t,x}(\R\times\R^3)$ spacetime norm, which suffices to guarantee scattering.   All proofs of these facts (see \cite{borg:scatter,CKSTT:gwp,KenigMerle:H1,KV:gopher}) rely on variants of the induction on energy procedure.  This procedure allows one to show that failure of global well-posedness and scattering guarantees the existence of special solutions that are well-localized in both position and momentum (Fourier) variables at each time. The proof of spatial localization uses the dispersive estimate in an essential way.

In \cite{KVZ:convex} we prove global well-posedness and scattering for the defocusing energy-critical NLS in the exterior of a smooth compact strictly convex obstacle in $\R^3$ for all initial data in the energy space.  This relies heavily on the results of this paper.  As an example of how Theorem~\ref{T:main} and its consequences enter into the analysis of dispersive equations, we will discuss the local and stability theories for the energy-critical NLS in exterior domains.
The initial value problem for this equation takes the following form:
\begin{equation}\label{E:Ecrit nls}
i \partial_t u = -\Delta_\Omega u \pm |u|^{\frac{4}{d-2}} u \qtq{with} u(0,x)=u_0(x)\in \dot H^1_D(\Omega).
\end{equation}

There are two obstructions to a verbatim repetition of the well-known arguments in the Euclidean setting: the gradient operator $\nabla$ does not commute with the propagator $e^{it\Delta_\Omega}$, while the operator $\sqrt{-\Delta_\Omega}$ does not obey the product/chain rules of regular calculus.  Theorem~\ref{T:main} provides the perfect tool to circumnavigate these issues: it allows us to freely pass from one operator to the other (provided $p<d$).

We start by making the notion of a solution precise:

\begin{definition}\label{D:solution}
Let $I$ be an open time interval containing the origin.  A function $u: I \times \Omega \to \C$ is a (strong) \emph{solution}
to \eqref{E:Ecrit nls} if it lies in the class $C^0_t \dot H^1_D(I\times \Omega) \cap L_t^{2(d+2)/(d-1)}L_x^{2d(d+2)/(d^2-2d-2)}(I\times\Omega)$ and satisfies
\begin{equation}\label{E:Ecrit duhamel}
u(t) = e^{it\Delta_\Omega} u_0  \mp i \int_0^t e^{i(t-s)\Delta_\Omega} |u(s)|^{\frac{4}{d-2}} u(s)\, ds,
\end{equation}
for all $t \in I$.
\end{definition}

\begin{remark}
There is considerable flexibility in the choice of spacetime norm in the definition above.  All can be shown to be equivalent \emph{a posteriori} by an application of the Strichartz inequality.
\end{remark}

We now show a strong form of local well-posedness.  We run a contraction mapping argument using Theorem~\ref{T:Strichartz}.  Of course, we have to choose our Strichartz spaces carefully to ensure that we can exploit Theorem~\ref{T:main}.  For the case of three dimensions, local well-posedness was obtained previously in \cite{BSS} and in \cite{IvanPlanch:IHP}.  In \cite{BSS} the authors prove an $L^4_tL^\infty_x$ Strichartz estimate, which allows them to rely only on the equivalence $H^1_D(\Omega)=H^1_0(\Omega)$.  In \cite{IvanPlanch:IHP} the authors prove multilinear estimates in Besov spaces; their results apply also for non-trapping obstacles in $\R^3$.

\begin{theorem}\label{T:LWP} Fix $d\geq 3$ and $\Omega\subset\R^d$ the exterior of a smooth compact strictly convex obstacle.  There exists $\eta>0$ such that if $u_0\in H^1_D(\Omega)$ obeys
\begin{equation}\label{E:LWP hyp}
\bigl\| \sqrt{-\Delta_\Omega} \; e^{it\Delta_\Omega} u_0 \bigr\|_{L_t^{\!\frac{2(d+2)}{d-1}} L_x^{\!\frac{2d(d+2)}{d^2+2}} (I\times\Omega) } \leq \eta
\end{equation}
for some time interval $I\ni 0$, then there is a unique strong solution to \eqref{E:Ecrit nls} on the time interval $I$; moreover,
\begin{equation}\label{E:LWP conc}
\bigl\| \sqrt{-\Delta_\Omega} \; u \bigr\|_{L_t^{\!\frac{2(d+2)}{d-1}} L_x^{\!\frac{2d(d+2)}{d^2+2}} (I\times\Omega) } \lesssim \eta.
\end{equation}
\end{theorem}

\begin{remarks}
1.  We do not use the symmetric Strichartz norm (with exponent $\frac{2(d+2)}{d}$) in this theorem because the counterexamples in \cite{HassellSikora, LSZ} show
that equivalence of norms \emph{fails} in this space when $d=3$.  (See also Section~\ref{SS:counter}.)

2. If $u_0$ has small $\dot H^1_D(\Omega)$ norm, then Theorem~\ref{T:Strichartz} guarantees that \eqref{E:LWP hyp} holds with $I=\R$.  Thus global well-posedness for small data is a corollary of this theorem.

3. For large initial data $u_0$, the existence of some small open interval $I\ni 0$ for which  \eqref{E:LWP hyp} holds follows from combining the monotone convergence theorem with Theorem~\ref{T:Strichartz}.  In this way, we obtain local well-posedness for all data in $H^1_D(\Omega)$.

4. The argument below holds equally well for initial data prescribed as $t\to\pm\infty$, thus proving the existence of wave operators.
\end{remarks}

\begin{proof} Consider the map $\Phi:u \mapsto \text{RHS\eqref{E:Ecrit duhamel}}$.   We will show this is a contraction on the ball
\begin{align*}
B:=\Bigl\{& u\in L_t^\infty H^{1}_D\cap L_t^{\frac{2(d+2)}{d-1}}H^{1,\frac{2d(d+2)}{d^2+2}}_D(I\times\Omega): \,
\|u\|_{L_t^\infty H^1_D} \leq 2 \|u_0\|_{H^1_D} + C(2\eta)^{\frac{d+2}{d-2}},\\
&\bigl\| \sqrt{-\Delta_\Omega} \; u \bigr\|_{L_t^{\frac{2(d+2)}{d-1}} L_x^{\frac{2d(d+2)}{d^2+2}}} \leq 2\eta, \qtq{and}
\bigl\| u \bigr\|_{L_t^{\frac{2(d+2)}{d-1}} L_x^{\frac{2d(d+2)}{d^2+2}}} \leq 2C\|u_0\|_{L_x^2}\Bigr\}
\end{align*}
under the metric given by
$$
d(u,v):= \| u -v \|_{L_t^{\frac{2(d+2)}{d-1}} L_x^{\frac{2d(d+2)}{d^2+2}}(I\times\Omega)}.
$$
The constant $C$ depends only on the dimension and the domain $\Omega$, and it reflects implicit constants in the Strichartz and Sobolev embedding inequalities, as well as those in Theorem~\ref{T:main} and Corollary~\ref{C:chain rule}.

Throughout the proof, all spacetime norms will be on $I\times\Omega$.  To see that $\Phi$ maps the ball $B$ to itself, we use the Strichartz inequality followed by Corollary~\ref{C:chain rule}, \eqref{E:LWP hyp}, Sobolev embedding (in the whole of $\R^d$), and then again Theorem~\ref{T:main}:
\begin{align*}
\bigl\|  &\sqrt{-\Delta_\Omega} \; \Phi (u) \bigr\|_{L_t^{\frac{2(d+2)}{d-1}} L_x^{\frac{2d(d+2)}{d^2+2}}}\\
&\leq \bigl\| \sqrt{-\Delta_\Omega} \; e^{it\Delta_\Omega} u_0 \bigr\|_{L_t^{\frac{2(d+2)}{d-1}} L_x^{\frac{2d(d+2)}{d^2+2}}}
+ C\bigl\| \sqrt{-\Delta_\Omega}\; (|u|^{\frac{4}{d-2}}u) \bigr\|_{L_t^{\frac{2(d-2)}{d-1}} L_x^{\frac{2d(d-2)}{d^2-6}}}\\
&\leq \eta + C\bigl\|  \sqrt{-\Delta_\Omega} \; u \bigr\|_{L_t^{\frac{2(d+2)}{d-1}} L_x^{\frac{2d(d+2)}{d^2+2}}}
	\|u\|_{L_t^{\frac{2(d+2)}{d-1}} L_x^{\frac{2d(d+2)}{d^2-2d-2}}}^{\frac4{d-2}}\\
&\leq \eta + C\bigl\| \sqrt{-\Delta_\Omega} \; u \bigr\|_{L_t^{\frac{2(d+2)}{d-1}} L_x^{\frac{2d(d+2)}{d^2+2}}} \bigl\| \nabla u \bigr\|_{L_t^{\frac{2(d+2)}{d-1}} L_x^{\frac{2d(d+2)}{d^2+2}}}^{\frac4{d-2}}\\
&\leq \eta + C \bigl\| \sqrt{-\Delta_\Omega}\;u\bigr\|_{L_t^{\frac{2(d+2)}{d-1}} L_x^{\frac{2d(d+2)}{d^2+2}}}^{\frac{d+2}{d-2}}\\
&\leq \eta + C(2\eta)^{\frac{d+2}{d-2}}\\
&\leq 2\eta,
\end{align*}
provided $\eta$ is chosen sufficiently small.

Similar estimates give
\begin{align*}
\bigl\| \Phi(u) \bigr\|_{L_t^{\frac{2(d+2)}{d-1}} L_x^{\frac{2d(d+2)}{d^2+2}}}
&\leq C \|u_0\|_{L_x^2} + C\bigl\||u|^{\frac{4}{d-2}}u \bigr\|_{L_t^{\frac{2(d-2)}{d-1}} L_x^{\frac{2d(d-2)}{d^2-6}}}\\
&\leq C \|u_0\|_{L_x^2} + C\| u\|_{L_t^{\frac{2(d+2)}{d-1}} L_x^{\frac{2d(d+2)}{d^2+2}}}
	\bigl\| \sqrt{-\Delta_\Omega}\;u\bigr\|_{L_t^{\frac{2(d+2)}{d-1}} L_x^{\frac{2d(d+2)}{d^2+2}}}^{\frac4{d-2}}\\
&\leq C \|u_0\|_{L_x^2} + C(2C\|u_0\|_{L_x^2})(2\eta)^{\frac4{d-2}}\\
&\leq 2C \|u_0\|_{L_x^2},
\end{align*}
and
\begin{align*}
\bigl\| \Phi(u) \bigr\|_{L_t^\infty H^1_D}
&\leq \|u_0\|_{H^1_D} + C\bigl\||u|^{\frac{4}{d-2}}u \bigr\|_{L_t^{\frac{2(d-2)}{d-1}} H_D^{1,\frac{2d(d-2)}{d^2-6}}}\\
&\leq  \|u_0\|_{H^1_D} + C\|u\|_{L_t^{\frac{2(d+2)}{d-1}} H_D^{1,\frac{2d(d+2)}{d^2+2}}}\bigl\| \sqrt{-\Delta_\Omega}\;u\bigr\|_{L_t^{\frac{2(d+2)}{d-1}} L_x^{\frac{2d(d+2)}{d^2+2}}}^{\frac4{d-2}}\\
&\leq \|u_0\|_{H^1_D} + C(2\eta + 2C\|u_0\|_{L_x^2})(2\eta)^{\frac4{d-2}}\\
&\leq 2 \|u_0\|_{H^1_D} + C(2\eta)^{\frac{d+2}{d-2}},
\end{align*}
provided $\eta$ is chosen small enough.

This shows that $\Phi$ maps the ball $B$ to itself.  Finally, to prove that $\Phi$ is a contraction, we argue as above:
\begin{align*}
d(\Phi(u),\Phi(v))&\leq C \bigl\| |u|^{\frac4{d-2}}u- |v|^{\frac4{d-2}}v\bigr\|_{L_t^{\frac{2(d-2)}{d-1}} L_x^{\frac{2d(d-2)}{d^2-6}}}\\
&\leq C d(u,v) \Bigl(\bigl\| \sqrt{-\Delta_\Omega}\;u\bigr\|_{L_t^{\frac{2(d+2)}{d-1}} L_x^{\frac{2d(d+2)}{d^2+2}}}^{\frac4{d-2}}+ \bigl\| \sqrt{-\Delta_\Omega}\;v\bigr\|_{L_t^{\frac{2(d+2)}{d-1}} L_x^{\frac{2d(d+2)}{d^2+2}}}^{\frac4{d-2}}\Bigr)\\
&\leq 2Cd(u,v)(2\eta)^{\frac4{d-2}}\\
&\leq \tfrac12 d(u,v),
\end{align*}
provided $\eta$ is chosen small enough.
 \end{proof}

Notice that in the preceding theorem, the initial data was taken to belong to the inhomogeneous space $H^1_D(\Omega)$.  The next result allows us to generalize the well-posedness result to initial data in the larger space $\dot H^1_D(\Omega)$ (which is the energy space), at least in spatial dimensions $3\leq d\leq 6$.  More importantly, it provides a key tool for the implementation of induction-on-energy/concentration-compactness in the paper
\cite{KVZ:convex}.

\begin{theorem}[Energy-critical stability result]\label{T:stab}
Fix $3\leq d\leq 6$ and $\Omega$ the exterior of a smooth compact strictly convex obstacle in $\R^d$.  Let $I$ a compact time interval and let $\tilde u$ be an approximate solution to \eqref{E:Ecrit nls} on $I\times \Omega$ in the sense that
$$
i\tilde u_t =-\Delta_\Omega \tilde u \pm |\tilde u|^{\frac4{d-2}}\tilde u + e
$$
for some function $e$.  Assume that
\begin{align*}
\|\tilde u\|_{L_t^\infty \dot H_D^1(I\times \Omega)}\le E \qtq{and} \|\tilde u\|_{L_{t,x}^{\frac{2(d+2)}{d-2}}(I\times \Omega)} \le L
\end{align*}
for some positive constants $E$ and $L$.  Let $t_0 \in I$ and let $u_0\in \dot H^1_D(\Omega)$ obey
\begin{align*}
\|u_0-\tilde u(t_0)\|_{\dot H_D^1(\Omega)}\le E'
\end{align*}
for some positive constant $E'$.  Assume also the smallness conditions
\begin{align*}
\bigl\|\sqrt{-\Delta_\Omega}\; e^{i(t-t_0)\Delta_\Omega}\bigl(u_0-\tilde u(t_0)\bigr)\bigr\|_{L_t^{\frac{2(d+2)}{d-2}}L_x^{\frac{2d(d+2)}{d^2+4}}(I\times \Omega)}
+\bigl\|\sqrt{-\Delta_\Omega}\; e\bigr\|_{N^0(I)}&\le\eps
\end{align*}
for some $0<\eps<\eps_1=\eps_1(E,E',L)$. Then, there exists a unique strong solution $u:I\times\Omega\mapsto \C$ to \eqref{E:Ecrit nls} with initial data $u_0$ at time $t=t_0$ satisfying
\begin{align*}
\|u-\tilde u\|_{L_{t,x}^{\frac{2(d+2)}{d-2}}(I\times \Omega)} &\leq C(E,E',L)\eps\\
\bigl\|\sqrt{-\Delta_\Omega}\;  (u-\tilde u)\bigr\|_{S^0(I\times\Omega)} &\leq C(E,E',L)\, E'\\
\bigl\|\sqrt{-\Delta_\Omega}\;  u\bigr\|_{S^0(I\times\Omega)} &\leq C(E,E',L).
\end{align*}
Here, $S^0$ denotes the intersection of any finite number of Strichartz spaces $L_t^q L_x^r$ with $(q,r)$ obeying the conditions of Theorem~\ref{T:Strichartz}, and $N^0$ denotes the sum of any finite number of dual Strichartz spaces $L_t^{q'} L_x^{r'}$.
\end{theorem}

The proof of this theorem follows the general outline in \cite{CKSTT:gwp,RV,TaoVisan}.  Small modifications are needed because these papers use the endpoint Strichartz inequality, which is unknown in exterior domains.  The proof of Theorem~\ref{T:LWP} shows the spaces that can be used in their place; moreover, these are spaces to which Theorem~\ref{T:main} applies.  We omit the details.

The analogue of Theorem~\ref{T:stab} for dimensions $d\geq 7$ is known in the Euclidean setting; see \cite{ClayNotes,TaoVisan}.  However, the proof relies on fractional chain rules for H\"older continuous functions and exotic Strichartz estimates.  Theorem~\ref{T:main} guarantees that the fractional chain rule can be imported directly from the Euclidean setting.  The exotic Strichartz estimates however are derived from the dispersive estimate; it is not known whether they hold in exterior domains.

\section{The multiplier theorem}\label{SS:multiplier}

In this section we prove the following analogue of the Mikhlin multiplier theorem:

\begin{theorem}[Multiplier theorem]\label{T:Mikhlin}
Suppose $m:[0,\infty)\to\C$ obeys
\begin{align}\label{E:symbol condition}
 \bigl| \partial^k m(\lambda)|\lesssim  \lambda^{-k} \qtq{for all} 0\leq k\leq \lfloor\tfrac d2\rfloor +1.
\end{align}
Then $m(\sqrt{-\Delta_\Omega})$, which we define via the $L^2$ functional calculus, extends uniquely from $L^p(\Omega)\cap L^2(\Omega)$ to a bounded operator on $L^p(\Omega)$, for
all $1<p<\infty$.
\end{theorem}

\begin{remark}
We have written our multipliers as functions of $\sqrt{-\Delta_\Omega}$ to emphasize the parallel to the Mikhlin theorem for Fourier multipliers.  If we choose a function $F$ so
that $F(\lambda^2)=m(\lambda)$ then the condition \eqref{E:symbol condition} is equivalent to $|F^{(k)}(\lambda^2)|\lesssim \lambda^{-2k}$
for $0\le k\le \lfloor\tfrac d2\rfloor +1$.
\end{remark}

Theorem~\ref{T:Mikhlin} is not new; it is subsumed by some comparatively general/abstract multiplier theorems. Most works in this direction have concentrated on concrete scenarios
(Lie groups and Schr\"odinger operators in $\R^d$, in particular) that almost include what is needed here, but not quite.  Nevertheless, the elegant techniques that have been
developed do allow for a concise and informative proof in our setting.  For this reason (and as a convenience to readers) we present a proof here.  Our presentation is primarily
influenced by \cite{Alexopoulos,SikoraWright}, which may be consulted for an introduction to earlier works.

The proof of Theorem~\ref{T:Mikhlin} is modelled on the classical Calder\'on--Zygmund argument for convolution operators $f\mapsto K*f$ discussed, for example, in
\cite[\S II.2--3]{Stein:small}.  The key condition besides boundedness in $L^2$ (which is equivalent to $|\hat K|\lesssim 1$) is that the kernel obeys
a cancellation condition.  Traditionally, this takes the form
\begin{align}\label{E:CZ conv cond}
\int_{|x-y|>2|y-y'|} |K(x-y)-K(x-y')|\,dx \lesssim 1 \qquad \text{uniformly for $y,y'\in\R^d$.}
\end{align}
The essence of this statement is that if $f\in L^1(\R^d)$ has mean zero, then $K*f$ is absolutely integrable away from the support of $f$, specifically,
on the set $\{ x : \dist(x,\supp(f)) \geq 2 \diam(\supp(f))\}$. Indeed, \eqref{E:CZ conv cond} corresponds to $f=\delta_y-\delta_{y'}$.

The next lemma represents the analogue of \eqref{E:CZ conv cond} in our setting.  The notion of zero mean needs to be adapted to the operator in question.
To better see the parallel, we note that in the usual Euclidean setting, $f=[1-e^{r^2\Delta}]\delta_y$ has mean zero and is morally supported on the ball $\{|x-y|<r\}$.

\begin{lemma}[Kernel bounds for multipliers]\label{L:Mikhlin kernel}
Let $m:[0,\infty)\to\C$ obey \eqref{E:symbol condition}.  Then
\begin{align}\label{E:Mikhlin kernel 0}
 \bigl\| m\bigl(\sqrt{-\Delta_\Omega}\bigr) \bigl[1-e^{r^2\Delta_\Omega}\bigr]\delta_y \bigr\|_{L^1(\{x\in\Omega : |x-y| > r\})}
  \lesssim 1
\end{align}
uniformly for $y\in \Omega$ and $r>0$.
\end{lemma}

\begin{proof}   Choose $\sigma\in\{\tfrac12,1\}$ so that $\lfloor\tfrac d2\rfloor +1=\frac d2+\sigma$.  Dividing the region $\{x: |x-y| > r\}$ into dyadic annuli where $|x-y|\sim R$, we see that it suffices to prove the following:
\begin{align}\label{E:Mikhlin kernel}
 \bigl\| m\bigl(\sqrt{-\Delta_\Omega}\bigr) \bigl[1-e^{r^2\Delta_\Omega}\bigr]\delta_y \bigr\|_{L^2(\{x\in\Omega : |x-y| > R\})}
  \lesssim \bigl(\tfrac rR\bigr)^{\frac\sigma2} R^{-\frac d2}
\end{align}
uniformly for $y\in \Omega$ and $R\geq r>0$.  The remainder of the proof is devoted to verifying this statement.

Let $a(\lambda):=m(\lambda)[1-e^{-r^2\lambda^2}]$, which we extend to all of $\R$ as an even function.  Elementary computations show that \eqref{E:symbol condition} implies
\begin{equation}\label{a multiplier bounds}
\bigl| \partial^k a(\lambda)|\lesssim  |\lambda|^{-k} \bigl[ 1 \wedge r|\lambda|\bigr] ^2\qtq{for all} 0\leq k\leq \tfrac d2 + \sigma.
\end{equation}

Next, we choose $\varphi\in C^\infty_c(\R)$ that is even, supported on $[-\frac12,\frac12]$, and obeys $\varphi(\tau)=1$ whenever $|\tau|<\frac14$. We then define $\check\varphi$
and $\check\varphi_R$ by
\begin{equation}
 {\check\varphi}_R(\lambda):= R {\check\varphi}(R\lambda) := \int e^{i\lambda\tau} \varphi\bigl(\tfrac\tau R\bigr)\,\tfrac{d\tau}{2\pi}.
\end{equation}
Taking the Fourier transform and then differentiating yields
\begin{equation}\label{perp to poly}
\varphi\bigl(\tfrac\tau R\bigr) =  \int e^{-i\lambda\tau} {\check\varphi}_R(\lambda)\,d\lambda \qtq{and}
    \int\lambda^\ell\check\varphi_R(\lambda)\,d\lambda = \delta_{0\ell}
\end{equation}
for all integers $\ell\geq 0$.

As both $a$ and $\varphi$ are even,
$$
a_1(\lambda) := (a * \check\varphi_R) (\lambda) = \int_0^\infty \tfrac1\pi\cos(\lambda\tau) \hat{a}(\tau) \varphi\bigl(\tfrac\tau R\bigr)\,d\tau,
$$
where $\hat a$ may be interpreted distributionally.  Noting that $\varphi(\tau/ R)$ is supported where $|\tau|\leq\frac{R}2$, finite speed of propagation for the wave equation
guarantees
$$
\supp\Bigl( a_1\bigl(\sqrt{-\Delta_\Omega}\bigr)\delta_{y} \Bigr) \subseteq \bigcup_{\tau\leq\frac R2} \supp\Bigl( \cos\bigl(\tau\sqrt{-\Delta_\Omega}\bigr)\delta_{y} \Bigr)
\subseteq \{ x\in\Omega : |x-y|\leq \tfrac12 R\}.
$$
Thus this part of the multiplier $a$ does not contribute to LHS\eqref{E:Mikhlin kernel}.  The use of finite speed of propagation for the wave equation as a tool for
estimating the kernels of spectral mulitipliers was introduced by M.~Taylor, \cite{Taylor:Laplace}.  For a proof of finite speed of propagation for the wave equation,
see \cite{Taylor:vol1}.

To control the remaining part of the multiplier, we will prove a pointwise bound on
$$
a_2(\lambda) := a_1(\lambda)-a(\lambda) = \int [a(\theta)-a(\lambda)] \check\varphi_R(\lambda-\theta)\,d\theta.
$$

When $|\lambda| \leq R^{-1}$ we use $|a(\theta)|\lesssim [1\wedge r|\theta|]^2$.  Combining this with the rapid decay of $\check\varphi$, we obtain
\begin{equation}\label{a_2 bound 1}
|a_2(\lambda)| \lesssim \bigl(\tfrac rR\bigr)^2\lesssim \bigl(\tfrac rR\bigr)^{\frac \sigma2}  \quad\text{when} \quad |\lambda| \leq R^{-1}.
\end{equation}

When $|\lambda| \geq R^{-1}$, we expand $a(\theta)$ in a Taylor series to order $k-1=\lfloor\frac d2\rfloor$; specifically,
$$
a(\theta)-a(\lambda)=P_k(\theta) + {\mathcal E}(\theta) \qtq{where} P_k(\theta) := \sum_{\ell=1}^{k-1} \frac{a^{(\ell)}(\lambda)}{\ell !} (\theta-\lambda)^\ell
$$
and ${\mathcal E}$ denotes the error, which we estimate using \eqref{a multiplier bounds} as follows:
\begin{equation*}
|{\mathcal E}(\theta)| \leq |a(\theta)| + |a(\lambda)| + |P_k(\theta)| \lesssim \bigl[1\wedge r|\lambda|\bigr]^2 \bigl|\tfrac{\theta-\lambda}{\lambda}\bigr|^k
    \quad\text{when}\quad |\theta-\lambda| > \tfrac12|\lambda|
\end{equation*}
and
\begin{equation*}
|{\mathcal E}(\theta)| \leq \bigl\| a^{(k)} \bigr\|_{L^\infty([\frac\lambda2,\frac{3\lambda}2])} |\theta - \lambda|^k
    \lesssim \bigl[1\wedge r|\lambda|\bigr]^2\bigl|\tfrac{\theta-\lambda}{\lambda}\bigr|^k
    \quad\text{when}\quad |\theta-\lambda| \leq \tfrac12|\lambda|.
\end{equation*}

From the orthogonality property \eqref{perp to poly}, we see that $P_k(\theta)$ makes no contribution to the convolution defining $a_2(\lambda)$.  Thus for $ |\lambda| \geq R^{-1}$,
\begin{equation}\label{a_2 bound 2}
|a_2(\lambda)| \lesssim  \int |{\mathcal E}(\theta)| |\check\varphi_R(\lambda-\theta)|\,d\theta \lesssim \bigl[1\wedge r|\lambda|\bigr]^2(R\lambda)^{-k}
\lesssim  \bigl(\tfrac rR\bigr)^{\frac \sigma2} \bigl(R|\lambda|\bigr)^{-\frac{d+\sigma}2}.
\end{equation}

Combining \eqref{a_2 bound 1} and \eqref{a_2 bound 2} and making some elementary manipulations shows
\begin{align*}
|a_2(\lambda)| &\lesssim \bigl[\tfrac rR\bigr]^{\frac\sigma2} (1+R^2\lambda^2)^{-\frac{d+\sigma}4}
    =  \frac{\bigl(\tfrac rR\bigr)^{\frac\sigma2}}{\Gamma(\frac{d+\sigma}4)} \int_0^\infty \bigl(\tfrac t{R^2}\bigr)^{\frac{d+\sigma}4} e^{-t/R^{2}} e^{-t\lambda^2}\tfrac{dt}{t}.
\end{align*}
The significance of the final representation is that the crude heat kernel estimate \eqref{E:crude heat bound} guarantees
$$
\| e^{t \Delta_\Omega} \delta_y \|_{L^2(\Omega)} \lesssim_d t^{-\frac d4}.
$$
Combining the two shows that
$$
\bigl\| a_2\bigl(\sqrt{-\Delta_\Omega}\,\bigr) \delta_y \bigr\|_{L^2(\Omega)}
    \lesssim_d \bigl(\tfrac rR\bigr)^{\frac\sigma2} R^{-\frac d2}\int_0^\infty \bigl(\tfrac t{R^2}\bigr)^{\frac\sigma4} e^{-t/R^{2}} \tfrac{dt}{t}
\lesssim_d \bigl(\tfrac rR\bigr)^{\frac\sigma2} R^{-\frac d2}.
$$
This completes the proof of Lemma~\ref{L:Mikhlin kernel}.
\end{proof}

\begin{proof}[Proof of Theorem~\ref{T:Mikhlin}]
We will prove that $T:=m(\sqrt{-\Delta_\Omega}\,)$ maps $L^1$ into weak-$L^1$.  Boundedness in $L^2$ follows from the spectral theorem.  Thus, by the Marcinkiewicz interpolation
theorem we deduce boundedness in $L^p$ for $1<p\leq 2$.  The result for $2<p<\infty$ then follows via duality.

Recall the Calder\'on--Zygmund decomposition (cf. \cite[\S I.3]{Stein:small}): Given $f\in L^1(\Omega)$ and $h>0$, there is a family of non-overlapping cubes
$\{Q_k\}\subset\R^d$ so that if we write $f=g+b$ with $b=\sum b_k$ and $b_k = \chi_{Q_k} f$, then
\begin{equation}\label{CZ decomp}
|g| \leq h \qquad\text{and}\qquad |Q_k| \leq \frac{1}{h} \int_{Q_k} |f(x)| \,dx \lesssim_d |Q_k|.
\end{equation}
We write $r_k$ for the radius ($\frac12\cdot$diameter) of $Q_k$ and $Q_k^*$ for the smallest concentric cube containing a ball of radius $2r_k$, that is, the $2\sqrt{d}$ dilate of
$Q_k$.

When treating usual Fourier multipliers, one alters this decomposition to make $b$ have mean zero on each cube $Q_k$.  To compensate, $g$ is altered to be a constant on each such cube.  In this setting, it is more natural to chose a notion of average that is adapted to the semigroup in question.  To this end, we define
$$
g_k := e^{r_k^2\Delta_\Omega} b_k \qtq{and}  \tilde b_k := \bigl[1 - e^{r_k^2\Delta_\Omega} \bigr] b_k .
$$

We now proceed in the usual manner:
\begin{equation*}
\big\{ |T f| > h \bigr\} \subseteq \big\{ |T g| > \tfrac13h \bigr\} \cup  \big\{ |T \, {\textstyle\sum} g_k| > \tfrac13h \bigr\}
    \cup \big\{|T \, {\textstyle\sum} \tilde b_k| > \tfrac13h  \bigr\}
\end{equation*}
and so, by Chebyshev's inequality and boundedness of $T$ in $L^2$,
\begin{equation*}
\bigl| \{ |T f| > h \} \bigr| \lesssim h^{-2} \bigl\|g \bigr\|_{L^2(\Omega)}^2 + h^{-2} \bigl\| {\textstyle\sum} g_k \bigr\|_{L^2(\Omega)}^2
    + h^{-1} {\textstyle\sum}\, \bigl\|  T \tilde b_k\bigr\|_{L^1(\Omega\setminus Q_k^*)} + {\textstyle\sum}\, |Q_k^*| .
\end{equation*}
Estimates on the first and last summands follow directly from \eqref{CZ decomp}:
$$
\|g\|_{L^2}^2 \leq \|g\|_{L^\infty}\|g\|_{L^1} \leq h \|f\|_{L^1}
\qtq{and}
\sum |Q_k^*| \lesssim_d  \sum \tfrac{1}{h} \int_{Q_k} \! |f(x)| \,dx \leq \tfrac{1}{h} \|f\|_{L^1}.
$$

Lemma~\ref{L:Mikhlin kernel} gives the needed bound on the $\tilde b_k$ terms, namely,
\begin{align*}
\bigl\|  T \tilde b_k\bigr\|_{L^1(\Omega\setminus Q_k^*)} &\lesssim \| b_k\|_{L^1}.
\end{align*}
By construction, $\sum |b_k| \leq |f|$ and so the above estimate can
be summed in $k$.

This leaves us to estimate $\| {\textstyle\sum} g_k
\|_{L^2(\Omega)}^2$.  Expanding the square and using the crude heat
kernel bound \eqref{E:crude heat bound},
\begin{align*}
\| {\textstyle\sum} g_k \|_{L^2(\Omega)}^2
&= \sum_{k,\ell} \langle b_k ,\; e^{(r_k^2+r_\ell^2)\Delta_\Omega} b_\ell \rangle \\
&\lesssim \sum_{r_k \geq r_\ell} \frac{1}{r_k^d}\int_{Q_\ell}\int_{Q_k} |b_k(x)| \, e^{- |x-y|^2/ 8 r_k^2} \, |b_\ell(y)| \,dx\,dy \\
&\lesssim \sum_{r_k \geq r_\ell} \frac{\|b_\ell\|_{L^1}}{|Q_\ell| r_k^d}  \int_{Q_\ell} \int_{Q_k} |b_k(x)| \, e^{ - |x-y'|^2/ 16 r_k^2}\, dx\,dy'.
\end{align*}
For the last step we used the fact that $|x-y|^2 \geq \tfrac12|x-y'|^2 - 4 r_l^2$ for any pair $y,y'\in\supp(b_\ell)=Q_\ell$.  From
\eqref{CZ decomp} we have $\int |b_\ell|\lesssim_d h |Q_\ell|$ and so
\begin{align*}
\| {\textstyle\sum} g_k \|_{L^2(\Omega)}^2
&\lesssim h \sum_{k} \frac{1}{r_k^d}  \int_{\Omega} \int_{Q_k} |b_k(x)| \, e^{ - |x-y'|^2/ 16 r_k^2} \, dx\,dy'\\
&\lesssim h \sum_{k}  \int_{Q_k} |b_k(x)|\, dx
\lesssim h \|f\|_{L^1}.
\end{align*}

Putting everything together, we obtain
$$
\bigl| \{ |T f| > h \} \bigr| \lesssim h^{-1} \|f\|_{L^1},
$$
which proves that $T$ maps $L^1$ into weak-$L^1$.  This completes the proof of Theorem~\ref{T:Mikhlin}.
\end{proof}

\section{Littlewood--Paley theory on exterior domains}\label{SS:LP}

In this section, we develop the basic ingredients of Littlewood--Paley theory adapted to the Dirichlet Laplacian on $\Omega$.
More precisely, we deduce the Bernstein inequalities and the square function inequalities from the multiplier theorem proved in
the previous section.

We will describe two kinds of Littlewood--Paley projections: one based on $C^\infty_c$ spectral multipliers and another based on the heat
kernel.  The latter version is much closer to the roots of the subject and will be used below to prove equivalence of Sobolev spaces.
We begin with the definition of the former.

Fix $\phi:[0,\infty)\to[0,1]$  a smooth non-negative function obeying
\begin{align*}
\phi(\lambda)=1 \qtq{for} 0\le\lambda\le 1 \qtq{and} \phi(\lambda)=0\qtq{for} \lambda\ge 2.
\end{align*}
For each dyadic number $N\in 2^\Z$, we then define
\begin{align*}
\phi_N(\lambda):=\phi(\lambda/N) \qtq{and} \psi_N(\lambda):=\phi_N(\lambda)-\phi_{N/2}(\lambda);
\end{align*}
notice that $\{\psi_N(\lambda)\}_{N\in \tz} $ forms a partition of unity for $(0,\infty)$.  With these functions in place,
we can now define our first family of Littlewood--Paley projections:
\begin{align*}
\po_{\le N} :=\phi_N\bigl(\sqrt{-\Delta_\Omega}\,\bigr), \quad \po_N :=\psi_N(\sqrt{-\Delta_\Omega}\,\bigr), \qtq{and} \po_{>N} :=I-\po_{\le N}.
\end{align*}

The second family of Littlewood--Paley projections we consider are
\begin{align*}
\tpo_{\le N} := e^{\Delta_\Omega/N^2}, \quad \tpo_N:=e^{\Delta_\Omega/N^2}-e^{4\Delta_\Omega/N^2}, \qtq{and} \tpo_{>N}:=I-\tpo_{\le N}.
\end{align*}

We will write $P_N$, $\tilde P_N$, and so forth, to represent the analogous operators associated to the usual Laplacian in the full Euclidean space.  Everything we discuss below applies equally well to these operators; indeed, we will be mimicking well-known proofs from the Euclidean setting.

We will show below that $f=\sum \po_N f = \sum \tpo_N f$ in $L^p$ sense; see Lemma~\ref{L:ident}.  This relies on some basic estimates for
these operators, which are included among the following:

\begin{lemma}[Bernstein estimates]\label{L:Bernie}
Fix $1<p<q\le \infty$ and $s\in\R$. Then for any $f\in C_c^{\infty}(\Omega)$, we have
\begin{gather*}
\|\po_{\le N} f \|_{\lpo}+\|\po_N f\|_{\lpo} + \|\tpo_{\le N} f \|_{\lpo}+\|\tpo_N f\|_{\lpo} \lesssim \|f\|_{\lpo},\\
\|\po_{\le N} f\|_{L^q(\Omega)}+\|\po_N f\|_{L^q(\Omega)} + \|\tpo_{\le N} f\|_{L^q(\Omega)}+\|\tpo_N f\|_{L^q(\Omega)}
    \lesssim N^{d(\frac 1p-\frac 1q)}\|f\|_{L^p(\Omega)},\\
N^s\|\po_N f\|_{\lpo} \sim \bigl\|(-\ld)^{\frac s2}\po_N f\bigr\|_{\lpo}.
\end{gather*}
\end{lemma}

\begin{proof} The first and third rows of estimates follow simply from Theorem~\ref{T:Mikhlin}.

For the second row of estimates, it suffices to observe that from the crude heat kernel bound \eqref{E:crude heat bound} and Young's convolution inequality we have
$$
\| e^{\Delta_\Omega / N^2} f\|_{L^q(\Omega)} \lesssim N^{d(\frac1p -\frac1q)}\|f\|_{L^p(\Omega)}.
$$
Note that both $\po_N$ and $\po_{\leq N}$ can be written as products of $\tpo_{\leq N}$ with multipliers that are $L^p$-bounded by virtue of Theorem~\ref{T:Mikhlin}.
\end{proof}

\begin{lemma}[Expansion of the identity]\label{L:ident}  For any $1<p<\infty$ and any $f\in L^p(\Omega)$,
\begin{equation}\label{E:LP id}
f(x) = \sum_{N\in 2^\Z} \bigl[\po_N f\bigr](x) = \sum_{N\in 2^\Z} \bigl[\tpo_N f\bigr](x),
\end{equation}
as elements of $L^p(\Omega)$.  In particular, the sums converge in $L^p(\Omega)$.
\end{lemma}

\begin{proof}
The Morawetz identity shows that $-\Delta_\Omega$ has purely absolutely continuous spectrum (cf. \cite{CyconFKS}); thus $0$ is not
an eigenvalue and so \eqref{E:LP id} holds when $p=2$.

For any $1<p<\infty$, Lemma~\ref{L:Bernie} shows that $\po_{\leq N}$, $\po_{\geq N}$, $\tpo_{\leq N}$, and $\tpo_{\geq N}$ are all bounded on $L^p(\Omega)$.
This guarantees that partial sums are bounded in $L^p$, and so it suffices to prove convergence for $f\in C^\infty_c(\Omega)$.

For $f\in C^\infty_c(\Omega)$ we can exploit convergence in $L^2$ and boundedness in all $L^q$ with $1<q<\infty$ to obtain convergence in $L^p$ via
interpolation (H\"older's inequality).
\end{proof}

\begin{theorem}[Square function estimates]\label{T:sq}
Fix $1<p<\infty$ and $s\geq 0$. Then for any $f\in C_c^{\infty}(\Omega)$,
\begin{align*}
\biggl\|\biggl(\sum_{N\in2^\Z} N^{2s}| \po_N f|^2\biggr)^{\!\!\frac 12}\biggr\|_{\lpo} \!\!
    \sim \bigl\|(-\ld)^{\frac s2} f \bigr\|_{\lpo}
        \sim \biggl\|\biggl(\sum_{N\in2^\Z} N^{2s}| (\tpo_N)^k f|^2\biggr)^{\!\!\frac 12}\biggr\|_{\lpo}
\end{align*}
provided the integer $k\geq 1$ satisfies $2k > s$.
\end{theorem}

\begin{proof} With Theorem~\ref{T:Mikhlin} in place, the proof of this result differs little from the argument in the Euclidean case (cf. \cite[\S IV.5]{Stein:small}).
We will just discuss the estimates for $\tpo_N$, the treatment of $\po_N$ being slightly simpler.

It suffices to prove that for all $g\in L^p(\Omega)$,
\begin{align*}
\| S(g) \|_{\lpo} \sim \| g \|_{\lpo} \qtq{where} S(g) := \biggl(\sum_{N\in2^\Z} N^{2s}| (\tpo_N)^k (-\ld)^{-\frac s2} g|^2\biggr)^{\!\!\frac 12}.
\end{align*}
Indeed, one merely has to apply this to $g=(-\Delta_\Omega)^{\frac s2}f$ with $f\in C^\infty_c(\Omega)$; to see that $g \in L^p(\Omega)$ we write
\begin{equation}\label{f hi lo}
g=(-\Delta_\Omega)^{\frac s2}f=(-\Delta_\Omega)^{\frac s2}P_{\leq 1}^\Omega f + (-\Delta_\Omega)^{\frac s2-k} P_{>1}^\Omega \bigl[(-\Delta_\Omega)^kf\bigr]
\end{equation}
and then invoke Theorem~\ref{T:Mikhlin}.

We begin by proving that $\| S(g) \|_{L^p} \lesssim \|g\|_{L^p}$.  First observe that
$$
N^{s} (\tpo_N)^k (-\ld)^{-\frac s2}  = m\bigl(\tfrac1N \sqrt{-\Delta_\Omega}\,\bigr) \qtq{with} m(\lambda) := \lambda^{-s} e^{-k\lambda^2} [1-e^{-3\lambda^2}]^k
$$
and moreover, $| \lambda^\ell \partial^\ell m (\lambda) | \lesssim_\ell \lambda^{2k-s} e^{-k\lambda^2}$ for all integers $\ell\geq 1$.  As a consequence, we
see that the multiplier
$$
m_\epsilon(\lambda) := \sum \epsilon_N m\bigl(\tfrac\lambda N\bigr) \qtq{obeys} | \lambda^\ell \partial^\ell m_\epsilon (\lambda) | \lesssim_\ell 1,
$$
uniformly in the choice of signs $\{\epsilon_N\}\subseteq\{\pm1\}$.  (Notice the summability relies on the restriction $s<2k$.)

Now consider statistically independent random signs $\epsilon_N$, each taking the values $\pm1$ with equal probability.  Applying Khintchine's inequality, Fubini, and then Theorem~\ref{T:Mikhlin}, we obtain
\begin{align*}
\int_\Omega \bigl|S(g)(x)\bigr|^p\,dx
\lesssim \int_\Omega \E\bigl\{ \bigl| [m_\epsilon g](x) \bigr|^p \bigr\} \,dx = \E \, \bigl\| [m_\epsilon g](x) \bigr\|_{\lpo}^p \lesssim \|g\|_{\lpo}^p.
\end{align*}
This proves $\| S(g) \|_{L^p} \lesssim \|g\|_{L^p}$.

We now consider the reverse inequality, which is proved via duality.  Define
\begin{align*}
\tilde m (\lambda) := \biggl( \sum_{N\in \tz} \bigl[ m\bigl(\tfrac\lambda N\bigr) \bigr]^2 \biggr)^{-1} ,
\end{align*}
where $m$ is the multiplier defined above.  It is not difficult to check that $\tilde m$ obeys the hypotheses of Theorem~\ref{T:Mikhlin} and so defines a bounded multiplier.
From Cauchy-Schwarz, the upper bound on the square function proved above, and this observation, we have
\begin{align*}
|\langle g ,\; h \rangle| &= \biggl| \sum_{N\in\tz} \Bigl\langle g, \;
        m\bigl(\tfrac1N \sqrt{-\Delta_\Omega}\,\bigr)^2 \tilde m\bigl(\sqrt{-\Delta_\Omega}\,\bigr) h \Bigr\rangle \biggr| \\
&= \biggl| \sum_{N\in\tz} \Bigl\langle m\bigl(\tfrac1N \sqrt{-\Delta_\Omega}\,\bigr) g, \;
        m\bigl(\tfrac1N \sqrt{-\Delta_\Omega}\,\bigr) \tilde m\bigl(\sqrt{-\Delta_\Omega}\,\bigr) h \Bigr\rangle \biggr| \\
&\leq \Bigl\langle S(g), \;  S\bigl(\tilde m\bigl(\sqrt{-\Delta_\Omega}\,\bigr) h\bigr) \Bigr\rangle \\
&\lesssim \bigl\| S(g) \bigr\|_{L^p} \bigl\| \tilde m\bigl(\sqrt{-\Delta_\Omega}\,\bigr) h \bigr\|_{L^{p'}} \\
&\lesssim \bigl\| S(g) \bigr\|_{L^p} \bigl\| h \bigr\|_{L^{p'}}
\end{align*}
The inequality $\|g\|_{L^p} \lesssim \| S(g) \|_{L^p}$ now follows by optimizing over $h\in L^{p'}$.
\end{proof}

The last result for this section will be used in Section~\ref{SS:equiv} to deduce Corollary~\ref{C:Riesz} from Theorem~\ref{T:main}.

\begin{lemma}\label{L:density}
For $1<p<\infty$ and $s<1+\frac1p$, $\{(-\Delta_\Omega)^{s/2} f : f\in C^\infty_c(\Omega)\}$ is dense in $L^p(\Omega)$. 
\end{lemma}

\begin{proof}
Fix $g\in L^p(\Omega)$ and $\eps>0$.  By Lemma~\ref{L:ident} there are $t_2>t_1>0$ so that
$$
\| g - (e^{2t_1\Delta_\Omega} - e^{2t_2\Delta_\Omega} ) g \|_{L^p(\Omega)} < \eps,
$$
which we may rewrite as
$$
\bigl\| g - (-\Delta_\Omega)^{s/2} (e^{t_1\Delta_\Omega} + e^{t_2\Delta_\Omega} ) (-\Delta_\Omega)^{-s/2} (e^{t_1\Delta_\Omega} - e^{t_2\Delta_\Omega} ) g \bigr\|_{L^p(\Omega)} < \eps.
$$
By Theorem~\ref{T:Mikhlin}, both
$$
(-\Delta_\Omega)^{s/2} (e^{t_1\Delta_\Omega} + e^{t_2\Delta_\Omega} ) \qtq{and} (-\Delta_\Omega)^{-s/2} (e^{t_1\Delta_\Omega} - e^{t_2\Delta_\Omega} )
$$
are bounded operators on $L^p(\Omega)$.  Correspondingly, there exists $k\in C^\infty_c(\Omega)$ so that
$$
\bigl\| g - (-\Delta_\Omega)^{s/2} (e^{t_1\Delta_\Omega} + e^{t_2\Delta_\Omega} ) k \bigr\|_{L^p(\Omega)} < 2\eps.
$$
In this way, it remains only to show that there exists $f\in C^\infty_c(\Omega)$ so that
$$
\bigl\| (-\Delta_\Omega)^{s/2} \bigl[ h - f \bigr] \bigr\|_{L^p(\Omega)} < \eps \qtq{where} h =(e^{t_1\Delta_\Omega} + e^{t_2\Delta_\Omega} ) k.
$$

By parabolic regularity, $h$ is smooth; we will construct $f$ by introducing a smooth cutoff.  This requires decay of $h$ near the boundary of $\Omega$
and at infinity; specifically, from Theorem~\ref{T:heat} we have $|h(x)|\lesssim \dist(x,\partial\Omega)$ and $|\Delta_\Omega h(x)|+ |h(x)|\lesssim_N |x|^{-N}$ for any $N$.
We will also use that these estimates imply $|\nabla h(x)|\lesssim_N |x|^{-N}$.

Let us suppose $0\in\Omega^c$.  Given $0<r\ll 1$ and $R\gg 1$ we may choose a cutoff function $\phi(x)$ which vanishes on the dilates $(1+r)\Omega^c$ and $2R\Omega$
and is equal to unity on $R\Omega^c \setminus (1+2r)\Omega$; moreover we may arrange that
$$
| \partial^\alpha_x \phi(x) | \lesssim_\alpha \dist(x,\partial\Omega)^{-|\alpha|}
$$
for all multi-indices $\alpha$.  Combining this with our prior information we deduce that
\begin{align*}
\bigl\| (-\Delta_\Omega)^{s/2} \bigl[ h - \phi h\bigr] \bigr\|_{L^p(\Omega)}
&\lesssim \bigl\| (1- \phi) h \bigr\|_{L^p(\Omega)}^{1-s/2} \bigl\| -\Delta_\Omega \bigl[(1- \phi) h\bigr] \bigr\|_{L^p(\Omega)}^{s/2} \\
&\lesssim_N \bigl( r^{1+\frac1p} + R^{-N} R^\frac{d}{p} \bigr)^{1-s/2} \bigl( r^{-1+\frac1p} + R^{-N} R^\frac{d}{p} \bigr)^{s/2}.
\end{align*}
As $s<1+\frac1p$, this can be made arbitrarily small by choosing $N>d/p$ and then $R$ large and $r$ small.  This completes the proof.
\end{proof}


\section{Riesz potential estimates and Hardy inequalities}\label{SS:Hardy}

The goal of this section is the proof of Hardy's inequality in exterior domains with respect to both the Euclidean Laplacian and the Dirichlet Laplacian; see Lemmas~\ref{L:Hardy Laplacian} and~\ref{L:Hardy Dirichlet} below.  As a stepping stone, we prove estimates on Riesz potentials, that is, the integral kernels associated to $(-\Delta_\Omega)^{-s/2}$; see Lemma~\ref{L:Riesz}.

The Riesz potentials are positive, as are the weights in Hardy's inequality.  Schur's test provides a powerful tool for determining the boundedness of integral operators with sign-definite kernels.  We will make use of the following formulation:

\begin{lemma}[Schur's test with weights]\label{L:Schur}
Suppose $(X, d\mu)$ and $(Y, d\nu)$ are measure spaces and let $w(x,y)$ be a positive measurable function defined on $X\times Y$.  Let $K(x, y):X\times Y\to \mathbb C$ satisfy
\begin{align}
\sup_{x\in X}\int_Y w(x,y)^{\frac 1p} |K(x,y)|\,d\nu(y)&=C_0<\infty,\label{E:Schur hyp1}\\
\sup_{y\in Y}\int_X w(x,y)^{-\frac1{p'}} |K(x,y)|\,d\mu(x)&=C_1<\infty,\label{E:Schur hyp2}
\end{align}
for some $1<p<\infty$. Then the operator defined by
\begin{align*}
Tf(x)=\int_{Y}K(x,y) f(y)\,d\nu(y)
\end{align*}
is a bounded operator from $L^p(Y, d\nu)$ to $L^p(X, d\mu)$. In particular,
\begin{align*}
\|Tf\|_{L^p(X,d\mu)}\lesssim C_0^{\frac 1{p'}}C_1^{\frac 1{p}}\|f\|_{L^p(Y,d\nu)}.
\end{align*}
\end{lemma}

\begin{remark}
This is essentially a theorem of Aronszajn.  When $K\geq 0$, Gagliardo has shown that the existence of a weight $w(x,y)=a(x)b(y)$ obeying \eqref{E:Schur hyp1} and \eqref{E:Schur hyp2} is necessary for the $L^p$ boundedness of $T$.  See the paper \cite{Gagliardo} of Gagliardo for further references.
\end{remark}

\begin{proof}
The proof is a simple application of H\"older's inequality.  Indeed, for $f\in L^p(Y,d\nu)$ and $g\in L^{p'}(X, d\mu)$, we have
\begin{align*}
\Bigl|\iint_{X\times Y} & K(x,y) f(y)g(x)\, d\nu(y)\, d\mu(x)  \Bigr|\\
&\lesssim \Bigl(\iint_{X\times Y} w(x,y)^{-\frac 1{p'}} |K(x,y)| |f(y)|^p\, d\nu(y)\, d\mu(x)\Bigr)^{\frac1p}\\
&\qquad\qquad \times\Bigl(\iint_{X\times Y} w(x,y)^{\frac 1p} |K(x,y)| |g(x)|^{p'}\, d\nu(y)\, d\mu(x)\Bigr)^{\frac1{p'}}\\
&\lesssim C_1^{1/p} \|f\|_{L^p(Y,d\nu)} C_0^{1/p'}\|g\|_{L^{p'}(X,d\mu)}.
\end{align*}
This yields the desired conclusion.
\end{proof}

\begin{lemma} [Riesz potentials] \label{L:Riesz}
Let $d\geq 3$ and $0<s<d$. Then the Riesz potential
\begin{align*}
(-\ld)^{-\frac s2}(x,y):=\frac 1{\Gamma(s/2)}\int_0^\infty t^{\frac s2}e^{t\ld}(x,y)\frac {dt}t
\end{align*}
satisfies
\begin{align*}
|(-\ld)^{-\frac s2}(x,y)|\lesssim \frac1{|x-y|^{d-s}}\biggl(\frac{\dist(x,\Omega^c)}{|x-y|\wedge\diam(\Omega^c)}\wedge1\biggr)
        \biggl(\frac{\dist(y,\Omega^c)}{|x-y|\wedge \diam(\Omega^c)}\wedge 1\biggr),
\end{align*}
uniformly for $x,y\in \Omega$.
\end{lemma}

\begin{proof} The proof relies on the heat kernel estimates from Theorem~\ref{T:heat}.  To keep formulas within the margins, we use the shorthands
$d(x):=\dist(x,\Omega^c)$ and $\diam:=\diam(\Omega^c)$.

We start with the long-time contribution $t\geq \diam^2$.  By Theorem~\ref{T:heat},
\begin{align*}
\int_{\diam^2}^\infty t^{\frac s2} e^{t\ld}(x,y) \frac {dt}t
&\lesssim\biggl(\frac{d(x)}{\diam}\wedge1\biggr)\biggl(\frac{d(y)}{\diam}\wedge 1\biggr)\int_{\diam^2}^\infty e^{-\frac{c|x-y|^2}t}t^{-\frac{d-s}2}\frac{dt}t\\
&\lesssim\biggl(\frac{d(x)}{\diam}\wedge1\biggr)\biggl(\frac{d(y)}{\diam}\wedge1\biggr)\frac1{|x-y|^{d-s}}
        \int_0^{\frac{|x-y|^2}{\diam^2}}e^{-c\tau}\tau^{\frac{d-s}2}\frac{d\tau}\tau\\
&\lesssim \biggl(\frac{d(x)}{\diam}\wedge1\biggr)\biggl(\frac{d(y)}{\diam}\wedge 1\biggr)\frac1{|x-y|^{d-s}},
\end{align*}
which is acceptable.

We now turn to the short-time contribution $0<t\leq \diam^2$. Again, by Theorem~\ref{T:heat} we estimate
\begin{align*}
\int_0^{\diam^2} t^{\frac s2} e^{t\ld}(x, & y) \frac{dt} t
\lesssim\int_0^{\diam^2}\biggl(\frac{d(x)}{\sqrt t}\wedge1\biggr)\biggl(\frac{d(y)}{\sqrt t}\wedge 1\biggr)e^{-\frac{c|x-y|^2}t}t^{-\frac{d-s}2}  \frac{dt}t \notag\\
&\lesssim \frac1{|x-y|^{d-s}}\int_{\frac{|x-y|^2}{\diam^2}}^\infty\biggl(\frac{d(x)\sqrt\tau}{|x-y|}\wedge1\biggr)\biggl(\frac{d(y)\sqrt\tau}{|x-y|}\wedge 1\biggr)
        e^{-c\tau}\tau^{\frac{d-s}2} \frac{d\tau}{\tau},
\end{align*}
and so it remains to show that
\begin{align}\label{short}
\int_{\frac{|x-y|^2}{\diam^2}}^\infty\biggl(\frac{d(x)\sqrt\tau}{|x-y|}\wedge1\biggr)&\biggl(\frac{d(y)\sqrt\tau}{|x-y|}\wedge 1\biggr)
        e^{-c\tau}\tau^{\frac{d-s}2} \frac{d\tau}{\tau}\notag\\
&\lesssim \biggl(\frac{d(x)}{|x-y|\wedge\diam}\wedge1\biggr)\biggl(\frac{d(y)}{|x-y|\wedge \diam}\wedge 1\biggr).
\end{align}
By symmetry, we may assume $d(x)\le d(y)$.  To continue, we split the integral in \eqref{short} into three parts:
\begin{align*}
\int_{\frac{|x-y|^2}{\diam^2}}^{\infty}
&=\int_{\frac{|x-y|^2}{\diam^2}}^{\frac{|x-y|^2}{d(y)^2}}+\int_{\frac{|x-y|^2}{d(y)^2}}^{\frac{|x-y|^2}{d(x)^2}}+\int_{\frac{|x-y|^2}{d(x)^2}}^\infty:=I+II+III.
\end{align*}

We begin with the contribution of $III$.  For these values of $\tau$ we have $1\leq \frac{d(x)\sqrt\tau}{|x-y|}\leq \frac{d(y)\sqrt\tau}{|x-y|}$, and so
\begin{align*}
III=\int_{\frac{|x-y|^2}{d(x)^2}}^\infty e^{-c\tau}\tau^{\frac{d-s}2}\frac{d\tau}\tau
\lesssim e^{-\frac{c|x-y|^2}{2d(x)^2}}
&\lesssim\biggl(\frac{d(x)}{|x-y|}\wedge 1\biggr)^2 \lesssim \text{RHS\eqref{short}}.
\end{align*}

Next, we consider the contribution of $I$.  On this region, we have $\frac{d(x)\sqrt\tau}{|x-y|}\leq \frac{d(y)\sqrt\tau}{|x-y|}\leq 1$.
The analysis of this term breaks into two additional cases: $|x-y|<d(y)$ and $|x-y|\geq d(y)$.  If $|x-y|<d(y)$, then
\begin{align*}
I=\frac{d(x)d(y)}{|x-y|^2}\int_{\frac{|x-y|^2}{\diam^2}}^{\frac{|x-y|^2}{d(y)^2}}e^{-c\tau}\tau^{\frac{d-s}2}\,d\tau
&\lesssim \frac{d(x)d(y)}{|x-y|^2}\biggl(\frac{|x-y|}{d(y)}\biggr)^{d-s+2}\\
&\lesssim \frac{d(x)}{d(y)} \lesssim \frac{d(x)}{|x-y|}\wedge1\lesssim \text{RHS\eqref{short}}.
\end{align*}
If instead $d(y)\le |x-y|$, we simply estimate
\begin{align*}
I&\lesssim \frac{d(x)d(y)}{|x-y|^2}
\lesssim\biggl(\frac{d(x)}{|x-y|}\wedge 1\biggr)\biggl(\frac{d(y)}{|x-y|}\wedge1\biggr)
\lesssim \text{RHS\eqref{short}}.
\end{align*}

Finally, we consider the contribution of $II$.  For these values of $\tau$ we have $\frac{d(x)\sqrt\tau}{|x-y|}\leq 1\leq \frac{d(y)\sqrt\tau}{|x-y|}$, and so
\begin{align*}
II=\frac{d(x)}{|x-y|}\int_{\frac{|x-y|^2}{d(y)^2}}^{\frac{|x-y|^2}{d(x)^2}} e^{-c\tau} \tau^{\frac{d-s+1}2}\frac{d\tau}\tau.
\end{align*}
We discuss three cases. When $|x-y|\le d(x)\le d(y)$, we obtain
\begin{align*}
II&\lesssim \frac{d(x)}{|x-y|}\biggl(\frac{|x-y|}{d(x)}\biggr)^{d-s+1}\lesssim 1\lesssim \text{RHS\eqref{short}}.
\end{align*}
When $d(x)\le |x-y|\le d(y)$, we simply have
\begin{align*}
II\lesssim \frac{d(x)}{|x-y|}\lesssim \text{RHS\eqref{short}}.
\end{align*}
In the remaining case, $d(x)\le d(y)\le |x-y|$, we estimate
\begin{align*}
II&\lesssim \frac{d(x)}{|x-y|} e^{-\frac{c|x-y|^2}{2d(y)^2}}\lesssim \frac{d(x)d(y)}{|x-y|^2}
\lesssim \text{RHS\eqref{short}}.
\end{align*}

This completes the proof of the lemma.
\end{proof}

\begin{lemma}[Hardy inequality for $\Delta_{\R^d}$]\label{L:Hardy Laplacian}
Fix $d\geq 3$, $1<p<\infty$, and $0<s<\frac dp$. Then for any $f\in C_c^{\infty}(\Omega)$, we have
\begin{align*}
\biggl\|\frac{f(x)}{\dist(x,\Omega^c)^s}\biggr\|_{\lpo}\lesssim \||\nabla|^s f\|_{\lpr}.
\end{align*}
\end{lemma}
\begin{proof}
The proof will be a simple application of Hardy's inequalities for bounded domains and for $\R^d$.  By translation, we may assume that $0\in\Omega^c$.
We will use the abbreviations $d(x)=\dist(x,\Omega^c)$ and $\diam=\diam(\Omega^c)$.

Let $\phi(x)\in C_c^{\infty}(\R^d)$ be a smooth bump function such that $\phi(x)=1$ for $x\in B(0, 2\diam)$ and
$\phi(x)=0$ if $x\in B(0, 3\diam)^c$. Notice that for $x\in \supp(1-\phi)$, we have $d(x)\sim |x|$.

We decompose $f(x)=\phi(x)f(x) + [1-\phi(x)]f(x)$ and treat the two pieces separately, beginning with the first.  Observe that $\phi f$ is supported in the
smooth bounded domain $U:=\Omega\cap B(0, 3\diam)$.  Therefore, by the Hardy inequality for such domains (cf. \cite[\S I.5.7]{Triebel:struc}) we obtain
\begin{align}\label{part1}
\biggl\|\frac{\phi(x)f(x)}{d(x)^s}\biggr\|_{\lpo}&\lesssim \biggl\|\frac{\phi(x)f(x)}{\dist(x, \partial U)^s}\biggr\|
    \lesssim \bigl\||\nabla|^s (\phi f)\bigr\|_{\lpr}.
\end{align}

When $0<s<1$, the fractional product rule (cf. \cite{ChW:fractional chain rule,Taylor:book}) and Sobolev embedding allow us to deduce
\begin{equation}\label{part 1a}
\begin{aligned}
\bigl\||\nabla|^s(\phi f) \bigr\|_{\lpr}
& \lesssim \|\phi\|_{L^\infty(\R^d)} \bigl\| |\nabla|^s f\bigr\|_{\lpr} +  \bigl\||\nabla|^s \phi\bigr\|_{L^{\frac ds}(\R^d)} \|f\|_{L^{\frac{pd}{d-sp}}(\R^d)}\\
&\lesssim \bigl\||\nabla|^s f\bigr\|_{\lpr}.
\end{aligned}
\end{equation}
Combining this with \eqref{part1} gives the needed bound on $\phi f$ for $0<s<1$.  When $s\geq1$ we write $s=k+\eps$ with $k$ an integer and $0\leq \eps<1$. Using
boundedness of Riesz transforms and the regular (pointwise) product rule, we have
\begin{align*}
\bigl\||\nabla|^s(\phi f) \bigr\|_{\lpr} & \lesssim
    \sum_{|\alpha|+|\beta| = k} \bigl\| |\nabla|^\eps \bigl[(\partial^\beta \phi)(\partial^\alpha f)\bigr]\bigr\|_{\lpr}.
\end{align*}
Using the fractional product rule together with Sobolev embedding and H\"older's inequality in much the same manner as before yields
$$
\bigl\| |\nabla|^s(\phi f) \bigr\|_{\lpr} \lesssim \bigl\||\nabla|^s f\bigr\|_{\lpr}.
$$
Combining this with \eqref{part1} provides the requisite estimate on $\phi f$.

To bound $[1-\phi] f$, we use the classical Hardy inequality on the whole of $\R^d$; this requires $s<\frac dp$.   (There are many proofs of this inequality;
our treatment of Region~IIa in the proof of Lemma~\ref{L:Hardy Dirichlet} shows how it can be deduced via Lemma~\ref{L:Schur}.)
Noting that $d(x)\sim |x|$ we have
\begin{align*}
\biggl\|\frac{[1-\phi(x)]f(x)}{d(x)^s}\biggr\|_{\lpo} &\lesssim \biggl\|\frac{[1-\phi(x)]f(x)}{|x|^s}\biggr\|_{\lpr} \\
    &\lesssim \bigl\| |\nabla|^s \bigl( [1-\phi]f\bigr)\bigr\|_{\lpr} \\
    &\lesssim \bigl\| |\nabla|^s f\bigr\|_{\lpr} + \bigl\| |\nabla|^s (\phi f)\bigr\|_{\lpr}.
\end{align*}
Combining this with the $L^p$ estimate on $|\nabla|^s(\phi f)$ discussed previously completes the proof of the lemma.
\end{proof}

\begin{lemma}[Hardy inequality for $\ld$]\label{L:Hardy Dirichlet}
Let $d\geq 3$, $1<p<\infty$, and $0<s<\min\{1+\frac1p,\frac dp\}$. Then for any $f\in
C_c^{\infty}(\Omega)$, we have
\begin{align*}
\biggl\|\frac {f(x)}{d(x)^s}\biggr\|_{\lpo}\lesssim \|(-\ld)^{\frac s2}f\|_{\lpo},
\end{align*}
where $d(x)=\dist(x, \Omega^c)$.
\end{lemma}

\begin{proof}
The claim follows from the following assertion, which we prove below:
\begin{align}\label{E:H2}
\biggl\|\frac1{d(x)^s}(-\Delta_\Omega)^{-\frac s2} g\biggr\|_{\lpo}\lesssim \|g\|_{\lpo} \qtq{for all} g\in L^p(\Omega).
\end{align}
Indeed, one merely has to apply this to $g=(-\Delta_\Omega)^{\frac s2}f$ with $f\in C^\infty_c(\Omega)$; such $g$ do indeed belong to $L^p(\Omega)$, as was
shown in the proof of Theorem~\ref{T:sq}.

By Lemma~\ref{L:Riesz}, to prove \eqref{E:H2} it suffices to show that the kernel
\begin{align*}
K(x,y):=\frac 1{d(x)^s}\frac 1{|x-y|^{d-s}}\biggl(\frac{d(x)}{|x-y|\wedge \diam}\wedge 1\biggr)\biggl(\frac{d(y)}{|x-y|\wedge \diam}\wedge 1\biggr)
\end{align*}
defines a bounded operator on $L^p(\Omega)$; here we again use the shorthand $\diam=\diam(\Omega^c)$.
Towards this goal, we subdivide $\Omega\times\Omega$ into several regions; in each of these regions, $L^p$ boundedness will be obtained via an application of Lemma~\ref{L:Schur} with a suitably chosen weight.  To begin, we subdivide into two main regions: $|x-y|\le \diam $ and $|x-y|> \diam$.

\textbf{Region I:} $|x-y|\le \diam$.  On this region, the kernel becomes
\begin{align*}
K(x,y)=\frac 1{d(x)^s}\frac1{|x-y|^{d-s}}\biggl(\frac{d(x)}{|x-y|}\wedge 1\biggr)\biggl(\frac{d(y)}{|x-y|}\wedge 1\biggr).
\end{align*}
To analyze this kernel, we further subdivide into four regions.

\textbf{Region Ia:} $|x-y|\le d(x)\wedge d(y)$.  Here, the kernel simplifies to
$$
K(x,y)=\frac1{d(x)^s|x-y|^{d-s}}
$$
and we also have
\begin{align*}
d(x)\le d(y)+|x-y|\le 2 d(y) \quad\text{and}\quad d(y)\le d(x)+|x-y|\le 2d(x).
\end{align*}
An easy computation then shows that
\begin{align*}
\int_{|x-y|\le d(x)}K(x,y)\,dy + \int_{|x-y|\le d(y)}K(x,y)\,dx\lesssim 1,
\end{align*}
and so $L^p$ boundedness on this region follows immediately from Lemma~\ref{L:Schur}.

\textbf{Region Ib:} $d(y)\le |x-y|\le d(x)$. On this region, the kernel takes the form
$$
K(x,y)=\frac{d(y)}{d(x)^s|x-y|^{d+1-s}}
$$
and it is easy to check that
\begin{align*}
\int_{d(y)\le |x-y|\le d(x)} K(x,y) \,dy&\lesssim \frac1{d(x)^s}\int_{|x-y|\le d(x)}\frac 1{|x-y|^{d-s}}\, dy\lesssim 1\\
\int_{d(y)\le |x-y|\le d(x)}K(x,y)\, dx&\lesssim d(y)\int_{|x-y|\ge d(y)}\frac 1{|x-y|^{d+1}}\, dx\lesssim 1.
\end{align*}
$L^p$ boundedness on this region follows again from an application of Lemma~\ref{L:Schur}.

\textbf{Region Ic:} $d(x)\le |x-y|\le d(y)$.  On this region, the kernel becomes
\begin{align*}
K(x,y)=\frac{d(x)^{1-s}}{|x-y|^{d+1-s}}
\end{align*}
and we have
\begin{align*}
|x-y|\le d(y)\le |x-y|+d(x)\le 2|x-y|.
\end{align*}
To prove $L^p$ boundedness on this region we use Lemma~\ref{L:Schur} with weight given by
\begin{align*}
w(x,y)=\biggl(\frac{d(x)}{|x-y|}\biggr)^{\alpha}\quad\text{with}\quad p(s-1)<\alpha<p'(2-s).
\end{align*}
The assumption $s<1+\frac1p$ guarantees that it is possible to chose such an $\alpha$.

That hypothesis \eqref{E:Schur hyp1} is satisfied in this setting follows from
\begin{align*}
\int_{d(x)\le |x-y|\le d(y)}w(x,y)^{\frac 1p} K(x,y) \,dy
&\lesssim \int_{|x-y|\ge d(x)}\frac{d(x)^{1-s+ \frac{\alpha}p}}{|x-y|^{d+1-s+\frac{\alpha}p}} \,dy\lesssim 1,
\end{align*}
while hypothesis \eqref{E:Schur hyp2} is deduced from
\begin{align*}
\int_{d(x)\le |x-y|\le d(y)}&w(x,y)^{-\frac 1{p'}}K(x,y) \,dx\\
&\lesssim \int_{\frac12 d(y)\leq|x-y|\leq d(y)}\frac {d(x)^{1-s-\frac\alpha{p'}}}{|x-y|^{d+1-s-\frac{\alpha}{p'}}} \,dx\\
&\lesssim d(y)^{-(d+1-s-\frac{\alpha}{p'})}\int_0^{d(y)}r^{1-s-\frac{\alpha}{p'}}d(y)^{d-1} \, dr\\
&\lesssim 1.
\end{align*}

\textbf{Region Id:} $d(x)\vee d(y)\le |x-y|$.  On this region, the kernel takes the form
\begin{align*}
K(x,y)=\frac{d(x)^{1-s} d(y)}{|x-y|^{d+2-s}}.
\end{align*}
We use again Lemma~\ref{L:Schur} with weight given by
$$
w(x,y)=\biggl(\frac{d(x)}{|x-y|}\biggr)^\alpha\quad\text{with}\quad p(s-1)<\alpha<p'(2-s).
$$
To verify hypothesis \eqref{E:Schur hyp1} in this setting, we estimate
\begin{align*}
\int_{d(x)\vee d(y)\le |x-y|}w(x,y)^{\frac 1p} K(x,y)\, dy
&\lesssim \int_{|x-y|\ge d(x)}\frac {d(x)^{1-s+\frac{\alpha}p}}{|x-y|^{d+1-s+\frac{\alpha}p}}\, dy \lesssim 1.
\end{align*}
Hypothesis \eqref{E:Schur hyp2} follows from
\begin{align*}
\int_{d(x)\vee d(y)\le |x-y|}& w(x,y)^{-\frac 1{p'}}K(x,y)\,dx\\
&\lesssim d(y)\int_{d(x)\vee d(y)\le |x-y|}\frac{d(x)^{1-s-\frac{\alpha}{p'}}}{|x-y|^{d+2-s-\frac{\alpha}{p'}}}\,dx\\
&\lesssim d(y)\sum_{R\ge d(y)}\frac 1{R^{d+2-s-\frac{\alpha}{p'}}}\int_0^{2R} r^{1-s-\frac\alpha{p'}}R^{d-1} dr\\
&\lesssim 1.
\end{align*}
In the display above, the sum is over $R\in2^\Z$.

We turn now to the contribution from the second main region.

\textbf{Region II:} $|x-y|>\diam$.  On this region, the kernel takes the form
\begin{align*}
K(x,y)=\frac 1{d(x)^s}\frac1{|x-y|^{d-s}}\biggl(\frac{d(x)}{\diam}\wedge 1\biggr)\biggl(\frac{d(y)}{\diam}\wedge 1\biggr).
\end{align*}
Without loss of generality, we may assume that the spatial origin is the centroid of $\Omega^c$.  To continue, we further subdivide into four regions.

\textbf{Region IIa:} $\diam\le d(x)\wedge d(y)$.   On this region we have
\begin{align*}
d(x)\sim |x| \quad \text{and} \quad \ d(y)\sim |y|
\end{align*}
and the kernel simplifies to
\begin{align*}
K(x,y)=\frac 1{d(x)^s|x-y|^{d-s}}\lesssim \frac 1{|x|^s|x-y|^{d-s}}.
\end{align*}

To prove $L^p$ boundedness on this region, we apply Lemma~\ref{L:Schur} with weight
$$
w(x,y)=\biggl(\frac{|x|}{|y|}\biggr)^{\alpha} \quad\text{with}\quad ps<\alpha< p'(d-s)\wedge pd.
$$

To verify hypothesis \eqref{E:Schur hyp1}, we estimate
\begin{align}\label{RIIa}
\int_{\diam\le d(x)\wedge d(y)}w(x,y)^{\frac 1p}K(x,y) \,dy\lesssim\int_{\R^d}\frac{|x|^{\frac\alpha{p}-s}}{|x-y|^{d-s}|y|^{\frac{\alpha}p}}\,dy.
\end{align}
Using that for the value of $\alpha$ chosen, $\frac\alpha{p}<d$, we estimate the contribution of the region $\{y\in \R^d:\, |y|\le 2|x|\}$ to the left-hand side of \eqref{RIIa} as follows:
\begin{align*}
\int_{|y|\le 2|x|}\frac{|x|^{\frac\alpha{p}-s}}{|x-y|^{d-s}|y|^{\frac{\alpha}p}}\, dy
&\lesssim \int_{|y|\le 2|x|}\frac {dy}{|x-y|^{d-\frac\alpha{p}}|y|^{\frac\alpha{p}}}  +\int_{|y|\le 2|x|}\frac {dy}{|x-y|^{d-s}|y|^s}
\lesssim 1,
\end{align*}
where in order to obtain the last inequality we consider separately the cases $|y|\leq \frac{|x|}2$ and $\frac{|x|}2<|y|\leq 2|x|$; in the former case note that $|x-y|\geq \frac{|x|}2$, while in the latter we have $|x-y|\leq 3|x|$.

To estimate the contribution of the region $\{y\in \R^d:\,|y|>2|x|\}$ to the left-hand side of \eqref{RIIa}, we use the fact that in this case, $|x-y|\sim |y|$; we obtain
\begin{align*}
\int_{|y|>2|x|}\frac{|x|^{\frac\alpha{p}-s}}{|x-y|^{d-s}|y|^{\frac{\alpha}p}}\,dy
&\lesssim |x|^{\frac\alpha p-s}\int_{|y|>2|x|}\frac1{|y|^{d-s+\frac\alpha{p}}} \,dy
\lesssim 1.
\end{align*}

Next, we verify that hypothesis \eqref{E:Schur hyp2} is satisfied in this case.  We have
\begin{align}\label{RIIa'}
\int_{\diam\le d(x)\wedge d(y)}w(x,y)^{-\frac 1{p'}}K(x,y) \,dx\lesssim\int_{\R^d}\frac{|y|^{\frac\alpha{p'}}}{|x|^{s+\frac\alpha{p'}}|x-y|^{d-s}}\, dx.
\end{align}
Using that for the value of $\alpha$ chosen, $s+\frac{\alpha}{p'}<d$, we estimate the contribution of the region $\{x\in \R^d:\, |x|\le 2|y|\}$ to the left-hand side of \eqref{RIIa'} as follows:
\begin{align*}
\int_{|x|\le 2|y|}&\frac{|y|^{\frac\alpha{p'}}}{|x|^{s+\frac\alpha{p'}}|x-y|^{d-s}}\, dx\\
&\qquad \lesssim \int_{|x|\le 2|y|}\frac{dx}{|x|^s|x-y|^{d-s}} + \int_{|x|\le 2|y|}\frac{dx}{|x|^{s+\frac\alpha{p'}}|x-y|^{d-s-\frac\alpha{p'}}}
\lesssim 1,
\end{align*}
where in order to obtain the last inequality we consider separately the cases $|x|\leq \frac{|y|}2$ and $\frac{|y|}2<|x|\leq 2|y|$.

Finally, to estimate the contribution of the region $\{x\in \R^d:\, |x|>2|y|\}$ to the left-hand side of \eqref{RIIa'}, we use the fact that in this case, $|x-y|\sim |x|$ to obtain
\begin{align*}
\int_{|x|>2|y|}\frac{|y|^{\frac\alpha{p'}}}{|x|^{s+\frac\alpha{p'}}|x-y|^{d-s}}\, dx
\lesssim |y|^{\frac \alpha{p'}}\int_{|x|>2|y|}\frac1{|x|^{d+\frac\alpha{p'}}} \,dx\lesssim 1.
\end{align*}

\textbf{Region IIb:} $d(y)\le \diam\le d(x)$.  On this region, $d(x)\sim |x|$, and so
\begin{align*}
K(x,y)=\frac{d(y)}{d(x)^s|x-y|^{d-s}\diam}\lesssim \frac 1{|x|^s|x-y|^{d-s}}.
\end{align*}
This upper bound is precisely the kernel considered in Region IIa.  The arguments given there show $L^p$ boundedness on this region.

\textbf{Region IIc:} $d(x)\le \diam \le d(y)$.  On this region, the kernel takes the form
\begin{align*}
K(x,y)=\frac {d(x)^{1-s}}{|x-y|^{d-s}\diam}.
\end{align*}
and we have
$$
|x|\sim \diam \quad \text{and}\quad |y|\sim d(y).
$$
Moreover, because on Region II we have $|x-y|>\diam$, we also obtain
$$
|y|\leq |x-y|+|x|\lesssim |x-y|+\diam\lesssim |x-y|.
$$

To obtain $L^p$ boundedness on this region, we apply Lemma~\ref{L:Schur} with weight
$$
w(x,y)=\frac{\diam^{\alpha_1} d(x)^{\alpha_2}}{|y|^\alpha}
$$
with $ps<\alpha=\alpha_1+\alpha_2<p'(d-s)$, $\alpha_1<p$, and $\alpha_2<p'(2-s)$.

We start by verifying hypothesis \eqref{E:Schur hyp1} in this setting:
\begin{align*}
\int_{d(x)\le \diam \le d(y)}w(x,y)^{\frac 1p}K(x,y) \,dy
&\lesssim \int_{d(x)\le \diam \le d(y)} \frac{\diam^{\frac{\alpha_1}p-1} d(x)^{1-s+\frac{\alpha_2}p}}{|y|^{\frac{\alpha}p}|y-x|^{d-s}}\,dy\\
&\lesssim {\diam}^{\frac{\alpha_1}p-1} d(x)^{1-s+\frac{\alpha_2}p} \int_{|y|\gtrsim d(x)}\frac1{|y|^{d-s+\frac\alpha{p}}}\,dy\\
&\lesssim {\diam}^{\frac{\alpha_1}p-1} d(x)^{1-\frac{\alpha_1}p}\\
&\lesssim 1.
\end{align*}

Next, we verify that hypothesis \eqref{E:Schur hyp2} holds in this setting:
\begin{align*}
\int_{d(x)\le \diam \le d(y)}w(x,y)^{-\frac 1{p'}}K(x,y) \,dx
&\lesssim \int_{d(x)\le \diam}\frac{d(x)^{1-s-\frac{\alpha_2}{p'}}|y|^{\frac\alpha{p'}}}{\diam^{1+\frac{\alpha_1}{p'}}|x-y|^{d-s}}\, dx\\
&\lesssim \int_{d(x)\le \diam}\frac{d(x)^{1-s-\frac{\alpha_2}{p'}}}{\diam^{1+\frac{\alpha_1}{p'}}|x-y|^{d-s-\frac{\alpha}{p'}}}\, dx\\
&\lesssim \int_{d(x)\le \diam}\frac{d(x)^{1-s-\frac{\alpha_2}{p'}}}{\diam^{d+1-s-\frac{\alpha_2}{p'}}}\, dx\\
&\lesssim {\diam}^{-(2-s-\frac{\alpha_2}{p'})} \int_0^{\diam} r^{{1-s-\frac{\alpha_2}{p'}}}\,dr\\
&\lesssim 1.
\end{align*}

\textbf{Region IId:} $d(x)\vee d(y)\le \diam$. On this region, we have $|x|\sim\diam$, $|y|\sim \diam$, and
$$
\diam<|x-y|\leq d(x)+d(y)+\diam\leq 3\diam.
$$
The kernel takes the form
\begin{align*}
K(x,y)=\frac{d(x)^{1-s}d(y)}{|x-y|^{d-s}\diam^2}\lesssim \frac{d(x)^{1-s}}{\diam^{d+1-s}}.
\end{align*}
On this region we apply Lemma~\ref{L:Schur} with weight given by
$$
w(x,y)=\biggl(\frac{d(x)}{|y|}\biggr)^\alpha \quad \text{with}\quad p(s-1)<\alpha<p'(2-s).
$$
To verify hypothesis \eqref{E:Schur hyp1} in this case, we simply estimate
\begin{align*}
\int_{d(x)\vee d(y)\le \diam}w(x,y)^{\frac 1p}K(x,y) \,dy
&\lesssim \int_{|y|\sim \diam}\frac{d(x)^{1-s+\frac\alpha{p}}}{|y|^{\frac\alpha{p}}\diam^{d+1-s}}\,dy
\lesssim 1.
\end{align*}
Similarly, that hypothesis \eqref{E:Schur hyp2} holds in this case follows from
\begin{align*}
\int_{d(x)\vee d(y)\le \diam}w(x,y)^{-\frac 1{p'}}K(x,y) \,dx
&\lesssim \int_{d(x)\le \diam}\frac{|y|^{\frac\alpha{p'}}d(x)^{1-s-\frac\alpha{p'}}}{\diam^{d+1-s}}\, dx\\
&\lesssim {\diam}^{\frac\alpha{p'}-(d+1-s)}\int_{d(x)\le \diam}d(x)^{1-s-\frac\alpha{p'}}\, dx\\
&\lesssim {\diam}^{-(2-s-\frac\alpha{p'})}\int_0^{\diam}r^{1-s-\frac\alpha{p'}}\, dr\\
&\lesssim 1.
\end{align*}

Putting everything together, we deduce that $K$ is the kernel of a bounded operator on $L^p(\Omega)$.  This completes the proof of Lemma~\ref{L:Hardy Dirichlet}.
\end{proof}


\section{Equivalence of Sobolev spaces}\label{SS:equiv}

In this section we prove Theorem~\ref{T:main}.  As discussed in the introduction, our approach to proving this theorem is to estimate the difference between the square functions discussed in Theorem~\ref{T:sq}, specifically, the square functions related to the heat kernel.  As a first step, we estimate the difference between the Littlewood--Paley projections.

\begin{lemma}\label{L:kernel} For an integer $k\geq 1$, let $K_N^k(x,y):=\bigl[(\tilde P_N)^k- (\tpo_N)^k\bigr](x,y)$. Then there
exists $c=c(k)>0$ such that
\begin{align*}
|K_N^k(x,y)|\lesssim_k N^d e^{-cN^2[d(x)^2+d(y)^2+|x-y|^2]},
\end{align*}
uniformly for $x\in \R^d$ and $y\in \Omega$.
\end{lemma}

\begin{proof} We write
\begin{align}\label{k0}
K_N^k(x,y)&=\bigl[e^{\Delta/N^2}-e^{4\Delta/N^2}\bigr]^{k}(x,y)-\bigl[e^{\ld/N^2}-e^{4\ld/N^2}\bigr]^{k}(x,y)\notag\\
&=\biggl\{\sum_{\ell=0}^{k-1}\bigl[e^{\Delta/N^2}-e^{4\Delta/N^2}\bigr]^{\ell}\bigl[e^{\Delta/N^2}-e^{\ld/N^2}-e^{4\Delta/N^2}+e^{4\ld/N^2}\bigr]\\
&\qquad \qquad\cdot \bigl[e^{\ld/N^2}-e^{4\ld/N^2}\bigr]^{k-\ell-1}\biggr\}(x,y).\notag
\end{align}
To proceed, we will estimate the kernels of each of the three factors appearing in the formula above.

Using the crude heat kernel estimate \eqref{E:crude heat bound}, a simple exercise with Gaussian integrals yields
\begin{align}\label{k1}
\sup_{0\le \ell\le k-1}\biggl|\bigl[e^{\ld/N^2}-e^{4\ld/N^2}\bigr]^\ell(x,y)\biggr|&+\biggl|\bigl[e^{\Delta/N^2}-e^{4\Delta/N^2}\bigr]^\ell(x,y)\biggr|\notag\\
&\lesssim_k N^d e^{-c_1N^2|x-y|^2}
\end{align}
for some $c_1=c_1(k)>0$.

We turn now to estimating the kernel of $e^{t\Delta}-e^{t\ld}$.  When $y\notin \Omega$, we have
\begin{align}\label{k2}
0\le [e^{t\Delta}-e^{t\ld}](x,y)=e^{t\Delta}(x,y) \lesssim t^{-\frac d2} e^{-\frac{|x-y|^2}{4t}}.
\end{align}

Consider now $y\in \Omega$.  As the obstacle $\Omega^c$ is convex, there exists a halfspace $\mathbb H_y\subset \Omega$ containing $y$ such that
\begin{align*}
\dist(y, \Omega^c)=\dist(y, \mathbb H_y^c).
\end{align*}
By the maximum principle, the heat kernel associated to the Dirichlet Laplacian on $\HH_y$ is positive and pointwise majorized by the heat kernel on $\Omega$, that is,
$$
0\leq e^{t\Delta_{\mathbb H_y}}(x,y)\leq  e^{t\ld}(x,y).
$$
For  $e^{t\Delta_{\mathbb H_y}}$ we have the exact formula
\begin{align*}
e^{t\Delta_{\mathbb H_y}}(x,y)=
\begin{cases}
e^{t\Delta}(x,y)-e^{t\Delta}(x,\bar y),& \mbox{ if } x\in \mathbb H_y,\\
0,&\mbox{ if } x\notin \mathbb H_y,
\end{cases}
\end{align*}
where $\bar y$ denotes the reflection of $y$ in the hyperplane $\partial \mathbb H_y$.  Thus, for $y\in \Omega$ we have
\begin{align}\label{k3}
0\le\bigl[e^{t\Delta}-e^{t\ld}\bigr](x,y)\le\bigl[e^{t\Delta}-e^{t\Delta_{\HH_y}}\bigr](x,y)
&=\begin{cases}
e^{t\Delta}(x,\bar y), & \mbox{ if } x\in \mathbb H_y\\
e^{t\Delta}(x,y), & \mbox{ if } x\notin \mathbb H_y
\end{cases}\notag\\
&\lesssim t^{-\frac d2} \exp \{-\tfrac{ d(x)^2+d(y)^2+|x-y|^2}{100t}\}.
\end{align}
To derive the last inequality, we argue as follows:  If $x\in \mathbb H_y$, then $x$ and $\bar y$ are on opposite sides of $\partial\HH_y$; thus $|x-\bar y|\geq |x-y|$, $|x-\bar y|\geq d(y)$, and
$$
|x-\bar y|\geq \bigl| x-\tfrac{y+\bar y}2\bigr|\geq d(x).
$$
Therefore $3|x-\bar y|\geq d(x)+d(y)+|x-y|$.  If instead $x\notin \mathbb H_y$, then $x$ and $y$ are on opposite sides of $\partial\HH_y$; thus $ |x-y|\ge d(y)$ and
$$
d(x)\leq d(y)+|x-y| \leq 2|x-y|.
$$
Therefore, $4|x-y|\geq d(x)+d(y)+|x-y|$.

Combining \eqref{k0}, \eqref{k1}, \eqref{k2}, and \eqref{k3}, for $x\in \R^d$ and $y\in \Omega$ we estimate
\begin{align}
|K_N^k(x,y)|
&\lesssim \iint_{\R^d\times\Omega^c}N^{3d}e^{-c_2N^2|x-x'|^2-c_2N^2|x'-y'|^2-c_2N^2|y-y'|^2}\,dx'\,dy'\notag\\
&+\iint_{\R^d\times\Omega} N^{3d}e^{-c_2N^2|x-x'|^2-c_2N^2[d(x')^2+d(y')^2+|x'-y'|^2]-c_2N^2|y'-y|^2} \,dx'\,dy'\label{int2}
\end{align}
for some $0<c_2\leq \min\{c_1,\frac1{100}\}$.

We now estimate the first integral appearing in \eqref{int2}.  As $y\in \Omega$, $|y-y'|\geq d(y)$. Also,
\begin{align*}
d(x)\leq |x-x'|+|x'-y'|+|y'-y|+d(y).
\end{align*}
Thus, we can bound the first integral by
\begin{align*}
N^{3d}e^{-c_3N^2[d(x)^2+d(y)^2]}&\iint_{\R^d\times\Omega}e^{-c_3N^2(|x-x'|^2+|x'-y'|^2+|y'-y|^2)}\,dx'\,dy'\\
&\lesssim N^d e^{-cN^2[d(x)^2+d(y)^2+|x-y|^2]}.
\end{align*}

To estimate the second integral in \eqref{int2}, we argue similarly and use $d(x)\le d(x')+|x-x'|$ and $d(y)\le d(y')+|y-y'|$ to obtain the bound
\begin{align*}
\iint_{\R^d\times\Omega}& N^{3d} e^{-c_2N^2|x-x'|^2-c_2N^2[d(x')^2+d(y')^2+|x'-y'|^2]-c_2N^2|y-y'|^2}\,dx'\,dy'\\
&\lesssim N^d e^{-cN^2[d(x)^2+d(y)^2+|x-y|^2]}.
\end{align*}

This completes the proof of the lemma.
\end{proof}

We are now ready to estimate the difference of the square functions.

\begin{proposition}\label{P:diff}  Fix  $d\geq 3$, $1<p<\infty$, and $s>0$. Then for any $f\in C_c^{\infty}(\Omega)$, we have
\begin{align*}
\biggl\|\biggl(\sum_{N\in 2^{\Z}}N^{2s}\bigl|(\tilde P_N)^kf\bigr|^2\biggr)^{\frac 12}
    -\biggl(\sum_{N\in 2^{\Z}}N^{2s}\bigl|(\pon)^kf\bigr|^2\biggr)^{\frac 12}\biggr\|_{L^p(\R^d)}
\lesssim_k \biggl\|\frac{f(x)}{d(x)^s}\biggr\|_{L^p(\Omega)}
\end{align*}
for any integer $k\geq 1$.
\end{proposition}

\begin{proof}
By the triangle inequality,
\begin{align*}
\text{LHS}&\lesssim \biggl\|\biggl(\sum_{N\in 2^{\Z}}N^{2s}\bigl|\bigl[(\tilde P_N)^k-(\pon)^k\bigr]f\bigr|^2\biggr)^{\frac12}\biggr\|_{\lpr}\\
&\lesssim \biggl\|\biggl(\sum_{N\in 2^{\Z}}N^{2s}\biggl|\int K_N^k(x,y)f(y)\,dy\biggr|^2\biggr)^{\frac 12}\biggr\|_{\lpr}\\
&\lesssim \biggl\|\sum_{N\in 2^{\Z}} N^s \int|K_N^k(x,y)||f(y)| \,dy\biggr\|_{\lpr}\\
&\lesssim \biggl\|\int\ \biggl(\sum_{N\in 2^{\Z}} N^s |K_N^k(x,y)|\biggr)\ |f(y)|\,dy\biggr\|_{\lpr}.
\end{align*}
Using Lemma~\ref{L:kernel}, we obtain
\begin{align*}
\sum_{N\in 2^{\Z}} N^s|K_N^k(x,y)|
&\lesssim_k \sum_{N\in 2^{\Z}}N^{d+s}e^{-cN^2[d(x)^2+d(y)^2+|x-y|^2]}\\
&\lesssim_k \sum_{N^2\leq [d(x)^2+d(y)^2+|x-y|^2]^{-1}} N^{d+s} \\
&\quad  +\sum_{N^2> [d(x)^2+d(y)^2+|x-y|^2]^{-1}} \frac{N^{d+s}}{\bigl(N^2[d(x)^2+d(y)^2+|x-y|^2]\bigr)^{d+s}}\\
&\lesssim_k \bigl[d(x)^2+d(y)^2+|x-y|^2\bigr]^{-\frac{d+s}2}.
\end{align*}
Therefore, to prove the proposition it suffices to show that the kernel $K:\R^d\times\Omega\to \R$ given by
\begin{align*}
K(x,y)=d(y)^s\bigl[d(x)^2+d(y)^2+|x-y|^2\bigr]^{-\frac{d+s}2}
\end{align*}
defines an operator bounded from $L^p(\Omega)$ to $L^p(\R^d)$. To establish this, we will use Lemma~\ref{L:Schur} with weight given by
$$
w(x,y)=\biggl(\frac{d(x)}{d(y)}\biggr)^{\alpha} \quad\text{and}\quad 0<\alpha< p'\wedge ps.
$$
We first verify that hypothesis \eqref{E:Schur hyp1} holds in this setting; indeed,
\begin{align*}
\int_{\Omega} w(x,y)^{\frac1 p}K(x,y)\, dy&=\int_{\Omega} \frac{d(x)^{\frac\alpha{p}}d(y)^{s-\frac{\alpha}{p}}}{[d(x)^2+d(y)^2+|x-y|^2]^{\frac{d+s}2}}\,dy\\
&\lesssim \int_{\Omega} \frac{d(x)^{\frac\alpha{p}}}{[d(x)+|x-y|]^{d+\frac\alpha{p}}}\,dy\lesssim 1,
\end{align*}
where in order to obtain the last inequality, we consider separately the regions $|y-x|\leq d(x)$ and $|y-x|>d(x)$.

Next we verify that hypothesis \eqref{E:Schur hyp2} holds in this setting.  We have
\begin{align}
\int_{\R^d} w(x,y)^{-\frac{\alpha}{p'}}K(x,y)\,dx=\int_{\R^d}\frac{d(y)^{s+\frac\alpha{p'}}}{d(x)^{\frac\alpha{p'}}[d(x)^2+d(y)^2+|x-y|^2]^{\frac{d+s}2}}\,dx.\label{xinte}
\end{align}
On the region where $|x-y|\le d(y)/2$, we have $d(x)\sim d(y)$. Thus, we may bound the contribution of this region by
\begin{align*}
d(y)^{-d}\int_{|x-y|\leq\frac12d(y)}\, dx\lesssim 1.
\end{align*}
The contribution of the region where $|x-y|> d(y)/2$ and $d(x)>d(y)$ is bounded by
\begin{align*}
d(y)^s\int_{|x-y|>d(y)/2}\frac 1{|x-y|^{d+s}}\,dx\lesssim 1.
\end{align*}
Finally, we estimate the contribution of the region where $|x-y|> d(y)/2$ and $d(x)\le d(y)$ to the right-hand side of \eqref{xinte} by
\begin{align*}
d(y)^{s+\frac{\alpha}{p'}}\sum_{R\ge d(y)}R^{-d-s}\int_{|x-y|\sim R}\frac{dx}{d(x)^{\frac\alpha{p'}}}
&\lesssim d(y)^{s+\frac{\alpha}{p'}}\sum_{R\ge d(y)}R^{-d-s}R^{d-\frac\alpha{p'}}
\lesssim 1.
\end{align*}
This completes the proof of the proposition.
\end{proof}

Finally, we are able to show the equivalence between the Sobolev
spaces.

\begin{proof}[Proof of Theorem~\ref{T:main}] Fix $f\in C_c^\infty(\Omega)$ and choose an integer $k\geq 1$ such that $2k>s$.  Using Theorem~\ref{T:sq}, the triangle inequality,
Proposition~\ref{P:diff}, and Lemma~\ref{L:Hardy Dirichlet}, we estimate
\begin{align*}
\||\nabla|^s f\|_p
&\sim \biggl\|\biggl(\sum_{N\in 2^{\Z}} N^{2s}\bigl|(\tilde P_N)^k f\bigr|^2\biggr)^{\frac 12}\biggr\|_p\\
&\lesssim \biggl\|\biggl(\sum_{N\in 2^{\Z}}N^{2s}\bigl|(\pon)^k f\bigr|^2\biggr)^{\frac12}\biggr\|_p
    +\biggl\|\biggl(\sum_{N\in 2^{\Z}}N^{2s}\bigl|\bigl[(\tilde P_N)^k-(\pon)^k\bigr] f\bigr|^2\biggr)^{\frac 12}\biggr\|_p\\
&\lesssim \|(-\ld)^{\frac s2} f\|_p+\biggl\|\frac{f(x)}{d(x)^s}\biggr\|_p\\
&\lesssim \|(-\ld)^{\frac s2} f\|_p.
\end{align*}
Arguing similarly and using Lemma~\ref{L:Hardy Laplacian} in place of Lemma~\ref{L:Hardy Dirichlet}, we obtain
\begin{align*}
\|(-\ld)^{\frac s2} f\|_p
&\sim \biggl\|\biggl(\sum_{N\in2^{\Z}}N^{2s}\bigl|(\pon)^k f\bigr|^2\biggr)^{\frac 12}\biggr\|_p\\
&\lesssim \biggl\|\biggl(\sum_{N\in 2^{\Z}}N^{2s}\bigl|(\tilde P_N)^kf\bigr|^2\biggr)^{\frac 12}\biggr\|_p
    +\biggl\|\biggl(\sum_{N\in 2^{\Z}} N^{2s}\bigl|\bigl[(\tilde P_N)^k-(\pon)^k\bigr] f\bigr|^2\biggr)^{\frac 12}\biggr\|_p\\
&\lesssim \||\nabla|^s f\|_p+\biggl\|\frac{f(x)}{d(x)^s}\biggr\|_p\\
&\lesssim \||\nabla|^s f\|_p.
\end{align*}
This completes the proof of the theorem.
\end{proof}

\begin{proof}[Proof of Corollary~\ref{C:Riesz}]
We must show that
\begin{equation}\label{g and k}
 \bigl\| (-\Delta_{\R^d})^{\frac s2} (-\Delta_{\Omega})^{-\frac s2} k \bigr\|_{L^p(\R^d)}  \lesssim \|k\|_{L^{p}(\Omega)}
\end{equation}
for a dense set of $k\in L^{p}(\Omega)$.  In view of Lemma~\ref{L:density}, we may choose $k=(-\Delta_{\Omega})^{s/2} f$ with $f \in C^\infty_c(\Omega)$.
Then, by Theorem~\ref{T:main},
\begin{align*}
\text{LHS\eqref{g and k}} = \bigl\| (-\Delta_{\R^d})^{\frac s2} f\bigr\|_{L^p(\R^d)} \lesssim  \bigl\| (-\Delta_{\Omega})^{\frac s2} f\bigr\|_{L^p(\Omega)}
	= \bigl\| k \bigr\|_{L^p(\Omega)},
\end{align*}
which proves \eqref{g and k} and so the corollary.
\end{proof}

\section{Counterexamples}\label{SS:counter}

The two propositions in this section prove that the conditions in Corollary~\ref{C:Riesz} are sharp, including the question of endpoints.

\begin{proposition}\label{P:1+1/p}
Corollary~\ref{C:Riesz} does not extend to $s \geq 1+\smash[b]{\frac1p}$.  More concretely, assuming $s \geq 1+\frac1p$, there exists $f\in L^p(\Omega)$ in the $L^p$ domain of $(-\Delta_\Omega)^{-s/2}$ and $\{g_n\} \subseteq C^\infty_c(\R^d)$ so that 
\begin{equation*}
\sup_n\|g_n\|_{L^{p'}(\R^d)} \lesssim 1 \qtq{but} \int_\Omega \bigl[(-\Delta_{\R^d})^{s/2}g_n\bigr](x) \bigl[(-\Delta_{\Omega})^{-s/2}f\bigr](x) \,dx \to \infty.
\end{equation*}
\end{proposition}

\begin{proof}
Choose a non-negative $h\in C^\infty_c(\Omega)$ and set $f = (-\Delta_\Omega)^{s/2} e^{\Delta_\Omega} h$.  Then $f\in L^p(\Omega)$ and $(-\Delta_\Omega)^{-s/2}f = e^{\Delta_\Omega} h\in L^p(\Omega)$.  Moreover, $e^{\Delta_\Omega} h$ is supported in $\bar\Omega$.  Now, toward a contradiction, suppose
\begin{equation*}
\int_\Omega \bigl[(-\Delta_{\R^d})^{s/2}g\bigr](x) \bigl[e^{\Delta_\Omega} h \bigr](x) \,dx \lesssim \|g\|_{L^{p'}(\R^d)} \quad\text{for all $g\in C^\infty_c(\R^d)$.}
\end{equation*}
By the Hardy inequality in \cite[\S1.5.7]{Triebel}, it follows that
\begin{equation}\label{hhardy}
\bigl\| \dist(x,\Omega^c)^{-s} e^{\Delta_\Omega} h \bigl\|_{L^p(\Omega)}  <\infty.
\end{equation}
This is easily falsified:  From the lower bounds on the heat kernel given in Theorem~\ref{T:heat}, we have $e^{\Delta_\Omega} h(x) \gtrsim \dist(x,\Omega^c)$ for
$x$ in bounded subsets of $\Omega$.  Thus $\text{LHS\eqref{hhardy}}=\infty$ precisely because $s \geq 1+\frac1p$.
\end{proof}

Next we show the necessity of the assumption $p<ds$.  Our argument gives further weight to the guiding principle ennunciated in \cite{GuiHassell}, namely, that boundedness of
Riesz transforms should be dictated by the absence of bounded non-constant zero-energy eigenfunctions.

\begin{proposition}\label{P:d/p}
The equivalence \eqref{E:equiv norms} fails when $d\geq 3$ and $d/p\leq s < 1+\tfrac1p$.  In particular, Corollary~\ref{C:Riesz} does not extend to such $s$.
\end{proposition}

\begin{proof}
We will construct an explicit one-parameter family of functions $f_R:\R^d\to\R$ that falsifies \eqref{E:equiv norms} in the case that $\Omega=\R^d\setminus B(0,1)$.  Henceforth, we consider only dyadic parameter values $R\gg 1$; all estimates will be uniform for $R$ sufficiently large.

The functions $f_R$ will not be compactly supported in $\Omega$ as required in \eqref{E:equiv norms}; however, they will be smooth and vanish at the boundary.  The arguments in the proof of Lemma~\ref{L:density} then show that functions $f_R$ may be approximated by functions in $C^\infty_c(\Omega)$ simultaneously in both topologies appearing in \eqref{E:equiv norms}.  Indeed, this can be done just using smooth cutoffs.  These approximation arguments are the source of the restriction $s<1+\frac1p$.

It is not difficult to construct smooth functions $\phi_R(x)$ so that
$$
\phi_R(x) = \log(R/|x|) \log^{-1}(R) \qtq{for $1\leq|x|\leq R/2$,} 
$$
$\phi_R(x) = 0$ for $|x|\geq R$, and
$$
\bigl| \partial_x^\alpha \phi_R(x) \bigr| \lesssim_\alpha  R^{-|\alpha|} \log^{-1}(R) \qtq{for} \tfrac R2\leq |x|\leq R
$$
and all multi-indices $\alpha$ with $|\alpha|\geq 0$.  We then define
$$
f_R(x) = \phi_R(x) \bigl[1 - |x|^{2-d}\bigr] \quad \text{for $x\in\Omega$} \qtq{and} f_R(x) = 0 \quad\text{otherwise}.
$$
Note that the term in square brackets is harmonic on $\Omega$ and vanishes on $\partial\Omega$.

From the explicit form of its construction, it is not difficult to see that
\begin{align*}
\bigl\|f_R\bigr\|_{L^p(\rho\leq|x|\leq 2\rho)} &\lesssim \rho^{\frac dp}   \\
\bigl\| \Delta_\Omega f_R \bigr\|_{L^p(\rho\leq|x|\leq 2\rho)} &\lesssim \rho^{\frac dp - 2} \log^{-1}(R) 
\end{align*}
uniformly for $1\leq\rho\leq R/2$.  Consequently,
\begin{equation*}
\bigl\| |x|^{-\frac dp} f_R\bigr\|_{L^p(\Omega)} \lesssim  \log^{\frac1p}(R)  \qtq{and}
\bigl\| |x|^{2-\frac dp} \Delta_\Omega f_R \bigr\|_{L^p(\Omega)} \lesssim \log^{\frac1p-1}(R).
\end{equation*}
This in turn gives us the bound
\begin{align}\label{E:loggy}
\bigl\| (-\Delta_\Omega)^{s/2}  f_R \bigr\|_{L^p(\R^d)} \lesssim \bigl\| |x|^{s-\frac dp}(-\Delta_\Omega)^{s/2}  f_R \bigr\|_{L^p(\R^d)}\lesssim \log^{\frac1p-\frac s2}(R)
\end{align}
via complex interpolation.  (Recall that the needed bounds on $(-\Delta_\Omega)^{it}$ follow from  Theorem~\ref{T:Mikhlin}.)

The key consequence of \eqref{E:loggy} is that $(-\Delta_\Omega)^{s/2} f_R\to 0$ in $L^p(\Omega)$ as $R\to\infty$. To complete the proof of the proposition, we need only show
\begin{align}\label{p/d lower}
\bigl\| (-\Delta_{\R^d})^{s/2}  f_R \bigr\|_{L^p(\R^d)} \gtrsim 1 \quad\text{uniformly in } R\gg1.
\end{align}
To this end, choose $\psi\in C^\infty_c(\R^d)$ so that $\psi(x)<0$ when $|x|<1$, $\psi(x)\geq 0$ when $|x|>1$, and
$$
\int_{\R^d} \psi(x) \, dx= 0 \qtq{and} \int_{\R^d} x \psi(x) \, dx= 0.
$$
(This is simplest if one takes $\psi$ radially symmetric.)

Noting that
$$
|x-y|^{s-d}= |x|^{s-d} + (d-s) \tfrac{x\cdot y}{|x|^{d+2-s}} + O\bigl(|x|^{s-d-2}\bigr) \qtq{for} |y|\lesssim 1
$$
and using the cancellation conditions on $\psi$, we obtain
$$
\biggl| \int_{\R^d} |x-y|^{s-d} \psi(y)\,dy  \biggr| \lesssim (1+|x|)^{s-d-2}.
$$
Consequently, $(-\Delta_{\R^d})^{-s/2}\psi\in L^{p'}(\R^d)$.  On the other hand, it is clear from the construction that $\int f_R(x) \psi(x)\,dx \gtrsim 1$.  Thus,
$$
1 \lesssim \langle\psi,f_R\rangle_{L^2(\R^d)} \lesssim \langle(-\Delta_{\R^d})^{-s/2}\psi,(-\Delta_{\R^d})^{s/2} f_R \rangle_{L^2(\R^d)}
	\lesssim_\psi  \bigl\| (-\Delta_{\R^d})^{s/2}  f_R \bigr\|_{L^p(\R^d)},
$$
which proves \eqref{p/d lower} and so completes the proof of the proposition.
\end{proof}

\section{Heat kernel regularity}\label{S:HKR}

In this section, we discuss for which values of $1<p<\infty$ one has
\begin{equation}\label{E:HKR}
 \bigl\| \sqrt{t} \,\nabla e^{t\Delta_\Omega} f \bigr\|_{L^p(\Omega)} \lesssim \| f \|_{L^p(\Omega)} \quad\text{uniformly for $t>0$.}
\end{equation}
We will give a sharp answer excepting the question of the endpoint.

The investigation of \eqref{E:HKR} on complete Riemannian manifolds and its connection to the boundedness of Riesz transforms forms the central topic of the paper \cite{ACDH}.  As noted in the introduction, it also forms a starting point in the analysis of multilinear Besov-space estimates in the work of Ivanovici--Planchon.  In this section, we prove

\begin{proposition}  For $d\geq 3$, the uniform estimate \eqref{E:HKR} holds for $p<d$ and fails for $p>d$.
\end{proposition}
 
\begin{proof}
The positive result follows immediately from the boundedness of Euclidean Riesz tranforms, Corollary~\ref{C:Riesz}, and Theorem~\ref{T:Mikhlin} together with the factorization
$$
\sqrt{t} \,\nabla e^{t\Delta_\Omega} = \bigl[  \nabla (-\Delta_{\R^d})^{-1/2} \bigr] \bigl[  (-\Delta_{\R^d})^{1/2}	(-\Delta_{\Omega})^{-1/2} \bigr] 
	\bigl[ (-t\Delta_\Omega)^{1/2} e^{t\Delta_\Omega} \bigr]. 
$$

We will prove the failure of \eqref{E:HKR} by explicit construction in the setting $\Omega=\R^d\setminus B(0,1)$.  For reasons of clarity, we proceed by contradiction, assuming that \eqref{E:HKR} did hold for some $p>d$.

Given $\eps>0$, let $g(t,x)$ denote the following solution of the heat equation in $\R^d$:
$$
g(t,x) = (1+\eps^2 t)^{-d/2} \exp\bigl\{ - \tfrac{\eps^2 |x|^2}{4(1+\eps^2 t)}\bigr\}.
$$
We also define $u(t,x)=(1-|x|^{2-d})g(t,x)$, which we claim is almost a solution to the Dirichlet heat equation in $\Omega$, at least for $0\leq t \leq \eps^{-2}$.  Indeed, for these values of $t$,
\begin{equation*}
(\partial_t - \Delta_\Omega) u(t,x) = -2(d-2) \tfrac{x}{|x|^d} \cdot \nabla g \qtq{and}
	\bigl\|\tfrac{x}{|x|^d} \cdot \nabla g(t,x)\bigr\|_{L^p(\Omega)} \lesssim  \eps^2.
\end{equation*}
Thus by Duhamel's formula and the assumption that \eqref{E:HKR} holds, we deduce that
\begin{align}
\bigl\| \sqrt{t}\, \nabla [e^{t\Delta_\Omega} u(0) - u(t)] \bigr\|_{L^p(\Omega)}
	&\lesssim \int_0^t \bigl\| \sqrt{t}\, \nabla e^{(t-s)\Delta_\Omega} \tfrac{x}{|x|^d} \cdot \nabla g(s) \bigr\|_{L^p(\Omega)} \, dt \notag\\
&\lesssim \int_0^t \sqrt{\tfrac{t}{t-s}} \, \bigl\| \tfrac{x}{|x|^d} \cdot \nabla g(s) \bigr\|_{L^p(\Omega)} \, dt \label{E:dudiff}\\
&\lesssim \eps^2 t. \notag
\end{align}

On the other hand, we have
$$
\nabla u(t,x) = (d-2) \tfrac{x}{|x|^d} g(t,x) - (1-|x|^{2-d}) \tfrac{\eps^2 x }{2(1+\eps^2t)} g(t,x),
$$
from which we see that if $\eps<2^{-d/2}$, then
$$
\bigl\| \sqrt{t}\, \nabla u(t,x) \bigr\|_{L^p(\Omega)} \gtrsim \bigl\| \sqrt{t}\,g(t,x) \bigr\|_{L^p(1<|x|<2)} \gtrsim \sqrt{t} (1+\eps^2 t)^{-d/2}.
$$
Combining this with \eqref{E:dudiff}, we deduce that
$$
\bigl\| \sqrt{t}\, \nabla e^{t\Delta_\Omega} u(0) \bigr\|_{L^p(\Omega)} \gtrsim \sqrt{t}
$$
uniformly for $\eps<2^{-d/2}$ and $0 \leq t \leq \eps^{-2}$.  However
$
\| u(0) \|_{L^p(\Omega)} \sim \eps^{-d/p}
$
and so choosing $t=\eps^{-2}$ and sending $\eps\to 0$ we see that \eqref{E:HKR} cannot hold when $p>d$.
\end{proof}


\end{document}